\documentclass[12pt]{article}

\usepackage[T1]{fontenc}
\usepackage[utf8]{inputenc}
\usepackage{graphicx}
\usepackage{amsmath, amsthm, amssymb, mathtools}
\usepackage{amsfonts}
\usepackage{bbm,bm}
\usepackage{xcolor}
\usepackage{hyperref}
\usepackage{natbib}
\usepackage{geometry}
\newcommand{\R}{\mathbb{R}}
\newcommand{\rough}[1]{\boldsymbol{{#1}}}

\theoremstyle{plain}
\newtheorem{theorem}{Theorem}[section]

\newtheorem{corollary}{Corollary}[theorem]

\theoremstyle{definition}
\newtheorem{definition}[theorem]{Definition}
\newtheorem{assumption}[theorem]{Assumption}

\theoremstyle{remark}
\newtheorem*{remark}{Remark}

\begin{document}

\title{Robust Filtering of L\'evy-driven Stochastic Models}

\author{
  Sharan Srinivasan\thanks{School of Electrical and Computer Engineering, Purdue University, West Lafayette, IN, USA. \texttt{srini256@purdue.edu}} \and
  Vijay Gupta\thanks{School of Electrical and Computer Engineering, Purdue University, West Lafayette, IN, USA. \texttt{gupta869@purdue.edu}} \and
  Harsha Honnappa\thanks{Edwardson School of Industrial Engineering, Purdue University, West Lafayette, IN, USA. \texttt{honnappa@purdue.edu}}
}

\date{}

\maketitle

\begin{abstract}
 We study robust nonlinear filtering for stochastic models driven by L\'evy processes, where the signal and observation processes are coupled through common Brownian and jump noise. Robustness, defined as the continuous dependence of the filter on the observation path, is essential whenever the observation process deviates from the idealized model, for instance when a path must be reconstructed from discrete-time samples. This question is well understood for continuous semimartingale systems but largely open in the presence of jumps.

We construct a version of the filter and establish its continuity in two regimes. For processes with finitely many jumps on compact intervals, we prove continuity in both the rough $p$-variation and $p$-variation topologies on cadlag path space, without requiring a separability condition on the jump coefficients. For processes with infinitely many jumps, we prove continuity in a modified rough $p$-variation topology adapted to cadlag geometric rough paths, under an additional separability assumption. In both cases, our approach relies on Stratonovich and Marcus flow decompositions rather than the It\^o-based methods of recent work. The resulting geometric rough-path lifts yield pathwise convergence guarantees and can be constructed directly from discrete observations without knowledge of the underlying probability law.
\end{abstract}


\noindent\textbf{Funding:} This research was supported by the Office of Naval Research (ONR) through grant no.\ FA9550-24-1-0210.

\section{Introduction}\label{sec:intro}

The problem of \textit{stochastic filtering} (see, e.g.~\citep{bain2009fundamentals,crisan2011oxford,kallianpur2013stochastic}) seeks to estimate an unobserved signal process $X_t$ from a noisy observation process $Y_t$. Given a measurable test function $f$, the \textit{filter} is the conditional expectation $\pi_t(f):=\mathbb{E}[f(X_t,Y_t)\vert\mathcal{F}^Y_t]$. As is well known, the filter obtained through the Kushner--Stratonovich equation is typically not unique. In practice, however, the observation path is never available exactly: one observes discrete-time samples and must reconstruct a continuous path, e.g.\ by linear or piecewise-constant interpolation. For the filter to be meaningful, it must be insensitive to such reconstruction choices. This leads to the \textit{robust filtering problem}: constructing a version $\theta^f_t(Y)$ of $\pi_t(f)$ that depends continuously on the observation path under a suitable topology.

The robust filtering problem has a long history in the continuous semimartingale setting. Clark~\citep{clark1978design} first advocated continuity of the filter in observation paths as a natural selection criterion. Clark and Crisan~\citep{clark} constructed a robust version $\theta_t^f(Y)$ when the signal and observation are driven by independent Brownian motions, using an integration-by-parts argument. The inclusion of common noise---when the signal and observation share a driving Brownian motion---introduces substantial additional difficulties. Davis~\citep{davis1987pathwise} addressed the one-dimensional case on manifolds using stochastic flow decompositions due to Kunita~\citep{kunita2006decomposition}, which require the driving vector fields to form a solvable Lie algebra. Elliott and Kohlmann~\citep{elliott1981robust} considered a related setting under the same algebraic constraint. Crisan et al.~\citep{crisanrough} removed the Lie-algebraic restriction by lifting the observation to a geometric rough path, reducing the problem to the continuity of rough differential equation (RDE) flows---a result that requires only Lipschitz regularity and linear growth of the vector fields.

All of the above works treat signal--observation systems driven by continuous semimartingales. Much less is known when the driving noise includes jumps. Qiao~\citep{qiaouncorr,qiaocorr} derived the Kushner--Stratonovich equations for signal--observation models driven by L\'evy-type noise (integer-valued random measures in the sense of~\citep{applebaum2009levy}), but did not address robustness. The contemporaneous work of Allan et al.~\citep{allan2025rough} is the first to construct a robust filter in the jump setting, developing new results on discontinuous RDEs and rough stochastic differential equations (RSDEs) based on non-geometric (It\^o) rough paths.

\subsection{Contributions}

We study the robust filtering problem for the signal--observation system
  \begin{subequations}\label{eq:general model}
    \begin{align}
        dX_t &:= b_1(t,X_t,Y_t)\,dt + \sigma_0(t,X_t,Y_t)\circ dB_t + \sigma_1(t,X_t,Y_t)\circ dW_t \\ \nonumber &\quad+ \int_{\mathbb{U}_p}f_1(t,X_{t-},Y_{t-},x)\,\tilde{N}_p(dt,dx) + \int_{\mathbb{U}_\lambda}f_3(t,X_{t-},Y_{t-},x)\,\tilde{N}_\lambda(dt,dx), \\
        dY_t &:= b_2(t,X_t,Y_t)\,dt+\sigma_2(t,Y_t)\circ dW_t + \int_{\mathbb{U}_\lambda}f_2(t,Y_{t-},x)\,\tilde{N}_\lambda(dt,dx),
    \end{align}
\end{subequations}
on a filtered probability space $(\Omega,\mathcal{F},(\mathcal{F}_t)_{t\in[0,T]},\mathbb{P})$ carrying independent Brownian motions $B$ (${d_B}$-dimensional) and $W$ (${d_Y}$-dimensional), and independent integer-valued random measures $N_p$, $N_\lambda$ on Blackwell spaces $(\mathbb{U}_p,\mathcal{U}_p)$, $(\mathbb{U}_\lambda,\mathcal{U}_\lambda)$ with $\mathbb{U}_d:=\{x\in\mathbb{R}^d:|x|<1\}$. Here $\tilde{N}_p$ and $\tilde{N}_\lambda$ denote the compensated measures with compensators $\nu_1(dx)\,dt$ and $\lambda(t,X_{t-},x)\nu_2(dx)\,dt$ respectively, where $\lambda:[0,T]\times\mathbb{R}^{d_X}\times\mathbb{U}_1\rightarrow (0,\infty)$ is Borel measurable and $\nu_1,\nu_2$ are $\sigma$-finite measures. The coefficient domains are $b_1:[0,T]\times\mathbb{R}^{d_X+d_Y}\rightarrow \R^{d_X}$, $b_2:[0,T]\times\mathbb{R}^{d_X+d_Y}\rightarrow \R^{d_Y}$, $\sigma_0:[0,T]\times\mathbb{R}^{d_X+d_Y}\rightarrow \R^{d_X\times d_B}$, $\sigma_1:[0,T]\times\mathbb{R}^{d_X+d_Y}\rightarrow \R^{d_X\times d_Y}$, $\sigma_2:[0,T]\times\R^{d_Y} \rightarrow \R^{d_Y\times d_Y}$, $f_1:[0,T]\times\mathbb{R}^{d_X+d_Y}\times \mathbb{U}_p\rightarrow \R^{d_X}$, $f_2:[0,T]\times\R^{d_Y}\times\mathbb{U}_\lambda\rightarrow \R^{d_Y}$, $f_3:[0,T]\times\R^{d_X+d_Y}\times\mathbb{U}_\lambda\rightarrow \R^{d_X}$, all Borel measurable. Note that the signal and observation share both Brownian and jump noise components.

Our approach is based on Stratonovich and Marcus flow decompositions. Specifically, we exploit the flows
\begin{equation}\label{eq:Flows of SDEs}
     \begin{split}
         \phi^\text{Strat}(t,x) &:= x + \int_0^tV(\phi^\text{Strat}(s,x))\circ dM_s,\\
         \phi^\text{Marcus}(t,x) &:= x + \int_0^tV(\phi^\text{Marcus}(s,x))\diamond dM_s,
     \end{split}
 \end{equation}
(where $M_t$ is a semimartingale and $V$ a vector field) which satisfy a Wong--Zakai theorem~\cite[Theorem 6.5]{kurtz1995stratonovich}. We construct a robust version of the filter and establish its continuity in two regimes:

\begin{itemize}
\item \textbf{Finite jump activity} (Section~\ref{sec:Finite}). When the jump measures have finitely many atoms on compact time intervals, we prove that the filter is Lipschitz continuous jointly in the Stratonovich rough-path lift of the observation and the counting-measure paths, with respect to the rough $p$-variation and $p$-variation topologies on cadlag path space. Crucially, this result does not require the separability condition on the jump coefficients imposed by \cite[Assumption~5.12]{allan2025rough}.
\item \textbf{Infinite jump activity} (Section~\ref{sec:infinite}). When the jump measures may have countably many atoms on compact intervals, we prove continuity of the filter in a modified rough $p$-variation topology ($\beta_p$) adapted to cadlag geometric rough paths constructed via the Marcus (log-linear) lift. In this regime we do require the separability assumption, for the same reasons as~\citep{allan2025rough}.
\end{itemize}

\subsection{Comparison with Allan et al.~\citep{allan2025rough}}

The work of Allan et al.~\citep{allan2025rough} is the closest to ours. They construct a robust filter using non-geometric (It\^o) rough paths and new results on discontinuous RDEs. Our geometric approach via Stratonovich and Marcus integrals leads to several structural differences:

\begin{enumerate}
\item \textit{Separability.} In~\citep{allan2025rough}, the separability condition \cite[Assumption~5.12]{allan2025rough} requiring the jump coefficients to factor as $f(t,X_t,Y_t,x)=\sum_i h^i(t,X_t,Y_t)\,g^i(t,x)$ is needed in both the finite and infinite jump regimes. While the Weierstrass approximation theorem guarantees that such factorisations are dense, there is no guarantee that the limit function yields a robust filter. We eliminate this assumption entirely in the finite-activity case (Section~\ref{sec:Finite}): the jump discontinuities contribute paths of finite variation that can be controlled directly, without forming a rough-path lift of the jump driver.

\item \textit{Pathwise convergence.} Since It\^o integrals converge only in probability and are non-anticipative, the approach of~\citep{allan2025rough} requires knowledge of the law of the driving noise. Our Stratonovich and Marcus integrals, being limits of ordinary integrals against smooth approximations (Wong--Zakai), provide pathwise convergence guarantees. The geometric rough-path lift of the observation can therefore be constructed directly from discrete samples without knowledge of the underlying probability measure.

\item \textit{Geometric structure.} The Stratonovich and Marcus lifts produce geometric rough paths~\citep{friz2012}, which satisfy the shuffle product identity and Chen's relation. Consequently, nonlinear functionals of the observation path admit representations as linear combinations of its signature~\citep{chevyrev2016primer}, opening a route to efficient numerical approximation of the filter. By contrast, the It\^o lift of~\citep{allan2025rough} yields non-geometric (branched) rough paths whose algebraic structure is considerably more involved~\citep{gubinelli2010ramification}. Furthermore, the flows generated by Stratonovich and Marcus integrals are diffeomorphisms~\cite[Proposition~4.3]{kurtz1995stratonovich}, making our framework directly extensible to signal--observation models on manifolds---a setting in which It\^o calculus is cumbersome due to the failure of the chain rule.
\end{enumerate}

\subsection{Organisation}

Section~\ref{sec:notation} collects notation and recalls the definitions of Stratonovich and Marcus integrals, geometric rough paths, and cadlag rough-path spaces. Section~\ref{sec:Robust Filter} formulates the robust filtering problem in generality via the Kallianpur--Striebel formula and a Girsanov change of measure. Section~\ref{sec:Flows} illustrates the flow-decomposition strategy in a scalar observation model with jumps, following~\citep{davis1987pathwise}. Section~\ref{sec:Finite} contains our main results in the finite-activity regime, and Section~\ref{sec:infinite} treats the infinite-activity case.

\section{Preliminaries}\label{sec:notation}

    We begin with some preliminaries that we will build on for our main results in Sections~\ref{sec:Finite} and~\ref{sec:infinite}.
    \subsection{Definitions}
    We define Stratonovich integrals and geometric rough paths first to emphasize our pathwise approach. For an in-depth discussion on  rough path theory, we refer the reader to resources such as~\citep{friz2012,friz2014course,friz2010multidimensional,geng2021introduction}.
    \begin{definition}
        Let $M$ and $N$ be two continuous semi-martingales. We denote the Stratonovich integral of $M$ with respect to $N$ as $\int_0^tM_s\circ dN_s$, and define it as
        \begin{equation*}
            \int_0^tM_s\circ dN_s = \int_0^t M_sdN_s+\frac{1}{2}\langle M,N\rangle_t,
        \end{equation*}
        where the integral on the right side is the It\^o integral.
    \end{definition}
    We now define geometric rough paths and Stratonovich lifts:
    \begin{definition}
        Equip the space of $\alpha$-H\"older rough paths $\mathcal{C}^\alpha([0,T],\mathbb{R}^d)$ with the metric
        \begin{equation*}
            \rho_\alpha(\rough{X},\tilde{\rough{X}}) := \sum_{n=1}^{N_\alpha}\sup_{0\leq s\leq t}\frac{|X^n_{s,t}-\tilde{X}^n_{s,t}|}{(t-s)^{n\alpha}},
        \end{equation*}
        where $N_\alpha = \lfloor1/\alpha\rfloor$.
        Then, the space of geometric rough paths is the closure of the rough lift of smooth paths in $\rho_\alpha$.
    \end{definition}
    We can now define the Stratonovich lift of a continuous semi-martingale:
    \begin{definition}
        Let $M$ be a continuous semi-martingale, its Stratonovich lift is defined as:
        \begin{equation}
            \rough{M}_t: = \left(M_t,\int_0^tM_s\otimes\circ dM_s\right).
        \end{equation}
    \end{definition}
    \begin{remark}
        A direct consequence of the definition of the Stratonovich lift and the Wong-Zakai theorem is that the Stratonovich lift of a semi-martingale is indeed a geometric rough path (e.g. see \cite[Proposition 3.6]{friz2014course}).
    \end{remark}
    We use the notion of Marcus differential equations in Section \ref{sec:infinite}. We define the solution of a Marcus SDE as in \citep{kurtz1995stratonovich,friz2012} which was introduced by Marcus in \citep{marcus2003modeling,marcus1981modeling}.
    First, let the jumps of a cadlag path $x$ be represented by
    \begin{equation*}
        \Delta x_t := x_{t-}-x_t,
    \end{equation*}
    where $x_{t-} := \lim_{s\downarrow t}x_s$.
    \begin{definition}
        Let $(\Omega,\mathcal{F},(\mathcal{F}_t)_{t\geq 0},\mathbb{P})$ be a filtered probability space, and $Z_t$ a $k$-dimensional semi-martingale. Let $f:\R^d\rightarrow \R^{d\times k}$ be a $C^1$ function, then we call $X_t$ the solution to the Marcus SDE:
        \begin{equation}\label{eq:marcus SDE}
            X_t = X_0 + \int_0^tf(X_t)\diamond dZ_t,
        \end{equation}
        if $X_t$ satisfies
        \begin{multline*}
            X_t = X_0 + \int_0^tf(X_{s-})dZ_s + \frac{1}{2}\int_0^tf'f(X_s)d[Z,Z]_s^c +\sum_{0\leq s\leq t}\left\{\phi(\Delta X_sf,Z_{s-})-Z_{s-}-f(Z_{s-})\Delta X_{s-}\right\},
        \end{multline*}
        where $[Z,Z]^c$ is the continuous part of the quadratic variation of the semi-martingale $Z$, and $\phi(g,x)$ is the time 1 solution to the ODE:
        \begin{equation*}
            \dot{y} = g(y),\hspace{10pt} y(0) = x.
        \end{equation*}
    \end{definition}

    \begin{remark}
        The advantage of using Marcus integration is that if $f\in C^3$, we recover the chain rule for cadlag semi-martingales (which both It\^o and Stratanovich integrals do not satisfy).
    \end{remark}

    \begin{remark}
        The Marcus integral reduces to the Stratonovich integral in the absence of jumps.
    \end{remark}

    Denote by $C([s,t],E)$ and $D([s,t],E)$ the space of continuous and c\'adl\'ag paths (respectively) from the interval $[s,t]$ to the metric space $(E,d)$. We recall two pseudo-metrics on these spaces.
    \begin{definition}
        Define the $p$-variation metric:
        \begin{equation}
            d_{p;[s,t]}(x,\bar{x}) := \|x-\bar{x}\|_{p-\text{var};[s,t]},
        \end{equation}
        and the Skorokhod metric
        \begin{equation}
            \sigma_{p;[s,t]}(x,\bar{x}) := \inf_{\lambda\in \Lambda_{[s,t]}}|\lambda|\vee d_{p;[s,t]}(x\circ\lambda,\bar{x}),
        \end{equation}
        where $\Lambda_{[s,t]}$ is the set of all bijections from $[s,t]$ to itself which are strictly increasing, and $|\lambda| := \sup_{u\in[s,t]}|\lambda(u)-u|$.
    \end{definition}
    Our goal is to convert elements in $D([0,T],E)$ to continuous paths by extending the path's jump times and interpolating with a continuous function. To this end, we now define path functions and the concept of admissibility.
    \begin{definition}
        A path function on $E$ is a map $ (a,b) \in J \mapsto \phi(a,b) \in  C([0,1],E)$
        where $J\subseteq E\times E$ is some subset, and $\phi(a,b)_0=a$ and  $\phi(a,b)_1=b$ for all $(a,b)\in J$. For some $x\in D([0,T],E)$ and a path function $\phi$, we call the pair $(x,\phi)$ admissible if $(x_{t-},x_{t})\in J$ for all jump times of $x$. The space of all admissible pairs of finite $p$-variation will be denoted $\mathcal{D}^{p-\text{var}}(E)$
    \end{definition}
    \begin{remark}
    These path functions are the interpolations we will use. In particular, when we have $E = G^N(\mathbb{R}^d)$, we will use the log-linear path function which corresponds to the Marcus integral. The log-linear path function is given by
    \begin{equation}\label{eq:loglin}
        \phi(a,b)_t = e^{t\log(a)+(1-t)\log(b)},
    \end{equation}
    where $\exp$ and $\log$ are taken in $T^N(\mathbb{R}^d)$.
    \end{remark}

    For any admissible pair $(x,\phi)$, define a summable sequence $\{r_n\}$ and let $r :=\sum_{k=1}^\infty r_k<\infty$, and order the jump times of $x$ by the size of the jump, i.e., $d(x_{t_1-},x_{t_1})>d(x_{t_2-},x_{t_2})>\cdots$. Define the increasing c\`adl\`ag function $\tau:[0,T]\rightarrow[0,T+r]$ by
    \begin{equation}\label{eq:time-ext}
        \tau(t) := t + \sum_{k=1}^\infty r_k \chi_{\{t_k\leq t\}},
    \end{equation}
    and the continuous path $\hat{x}\in C([0,T+r],E)$ by
    \begin{equation}\label{eq:con-interp}
        \hat{x}_s := \begin{cases}
            x_s & \text{if} \hspace{10pt} s = \tau(t) \text{ for some } t\in [0,T],\\
            \phi(x_{t_k-},x_{t_k})_{(s-\tau(t_{k-}))/r_k} &\text{if} \hspace{10pt} s\in[\tau(t_{k-}),\tau(t_k)) \text{ for some } k\in \mathbb{Z}_+.
        \end{cases}
    \end{equation}
    Let $\tau_r(t) = t(T+r)/T$ be the rescaling of domains, and define $x^\phi\in C^{p-\text{var}}([0,T],E)$ as
    \begin{equation}\label{eq:cont-from-discont}
        x^\phi = \hat{x}\circ \tau_r.
    \end{equation}
    Defining $\tau_x = \tau^{-1}_r\circ \tau$, one can obtain $x$ by $x = x^\phi\circ\tau_x$.

    We are now ready to define RDEs driven by general cadlag rough paths. We equip the free nilpotent Lie group $G^N(\mathbb{R}^d)$ with the Carnot-Carath\'eodory norm $\|\cdot\|_{CC}$ (see \cite[Section 2.5]{friz2012}) and its induced metric $d_{CC}$. The increments of the $G^N(\mathbb{R}^d)$ path are defined as $\rough{x}_{s,t} = \rough{x}_s^{-1}\rough{x_t}$, and the p-variation is given as
    \begin{equation}
        \rho_{p;[s,t]}(\rough{x},\rough{\hat{x}}) : = \max_{1\leq k\leq N}\sup_{\mathcal{P}\subset [s,t]}\left(\sum_{t_i\in \mathcal{P}}|\rough{x}_{t_i,t_{i+1}}^k-\rough{\hat{x}}_{t_i,t_{i+1}}^k|^{p/k}\right)^{k/p}
    \end{equation}

    Recall the path spaces  $D^{p-\text{var}}([0,T],G^N(\mathbb{R}^d))$ and $C^{p-\text{var}}([0,T],G^N(\mathbb{R}^d))$, i.e., the space of cadlag and continuous paths (respectively) with finite $p$-variation for some $1\leq p<N+1$.

    \begin{definition}\label{def:marcus-lift}
    Given a $\mathbb{R}^d$-valued path $x$, define its Marcus lift as:
    \begin{equation}\label{eq:marcus-lift}
        \mathbf{x}_{s,t} := \left(x_{t,s},\int_s^t x_{s,r}\otimes\diamond dx_r\right).
    \end{equation}
    \end{definition}
    It is easy to check that the Marcus lift satisfies:
    \begin{equation*}
        \log\Delta{\rough{x}}_t \in \mathbb{R}^d\oplus\{0\}.
    \end{equation*}

    \begin{definition}
        For an admissible pair $(\mathbf{x},\phi)\in \mathcal{D}^{p-\text{var}}(G^N(\mathbb{R}^d))$ and a family of vector fields $V = (V_1,\dots,V_d)$ in $\text{Lip}^{\gamma+m-1}(\mathbb{R}^e)$ for some $\gamma>p$ and $m\geq1$, let $\bar{y} \in C^{p-\text{var}}([0,T],\mathbb{R}^e)$ be the solution to the continuous RDE
        \begin{equation*}
            d\bar{y}_t = V(\bar{y}_t)d\mathbf{x}^\phi_t, \hspace{10pt} \bar{y}_0 = y_0\in \mathbb{R}^e.
        \end{equation*}
        Define the solution $y\in D^{p-\text{var}}([0,T],\mathbb{R}^e)$ to the canonical RDE
        \begin{equation}\label{eq:canonical-RDE}
            dy_t = V(y_t)\diamond d(\mathbf{x}_t,\phi), \hspace{10pt} y_0\in \mathbb{R}^e,
        \end{equation}
        as $y = \bar{y}\circ\tau_\mathbf{x}$.

        In particular, when $\phi$ is the log linear path defined in equation \eqref{eq:loglin}, we denote the RDE by
        \begin{equation*}
            dy_t = V(y_t)\diamond d\mathbf{x}_t, \hspace{10pt} y_0\in \mathbb{R}^e.
        \end{equation*}
    \end{definition}
    Note that the solution $\tilde{y}_t$ depends on the choice of series $\sum_kr_k$ and the path function $\phi$. On the other hand, the solution $y_t$ to equation \eqref{eq:canonical-RDE} is independent of this choice, since the map $\tau_\mathbf{x}$ essentially `cuts off' the added path sections used to fill-in the fictitious time. Next, we need to define a new metric to show continuity in rough spaces.

    \begin{definition}
        For $(\mathbf{x},\phi),(\mathbf{\tilde{x}},\tilde{\phi})\in \mathcal{D}^{p-\text{var}}([0,T],G^N(\mathbb{R}^d))$, define
        \begin{equation}
            \alpha_p(\mathbf{x},\mathbf{\tilde{x}}) := \lim_{\delta\rightarrow0}\sigma_{p}(\mathbf{x}^{\phi,\delta},\mathbf{\tilde{x}}^{\tilde{\phi},\delta}),
        \end{equation}
        and
        \begin{equation}
            \beta_p(\mathbf{x},\mathbf{\tilde{x}}) := \lim_{\delta\rightarrow0}\rho_{p}(\mathbf{x}^{\phi,\delta},\mathbf{\tilde{x}}^{\tilde{\phi},\delta}),
        \end{equation}
        where $\mathbf{x}^{\phi,\delta}$ is the same as that in equation \eqref{eq:cont-from-discont}, but with the series $\sum_k \delta\cdot r_k$ instead of $\sum_kr_k$ for some $\delta>0$.
    \end{definition}

   \subsection{Assumptions}
     For the system \eqref{eq:general model} to have a unique strong solution, we require the following assumptions on the coefficients:
     \begin{assumption}\label{as:lin boundedness}
         The coefficients have at most linear growth, i.e., there exists a constant $K>0$ such that:
         \begin{equation*}
         \begin{split}
             |b_1(t,x,y)|+|b_2(t,x,y)| + |\sigma_0(t,x,y)|+|\sigma_1(t,x,y)|+|\sigma_2(t,y)|\leq K(1+|x|+|y|),\\
             \|f_1(t,x,y,\cdot)\|_{L^2(\nu_1)}+\|f_2(t,y,\cdot)\|_{L^2(\nu_2)} + \|f_3(t,x,y,\cdot)\|_{L^2(\nu_2)}\leq K(1+|x|+|y|).
         \end{split}
         \end{equation*}
         for all $t\in[0,T]$, $x\in\R^{d_X}$ and $y\in \R^{d_Y}$.
     \end{assumption}

     \begin{assumption}\label{as:initial values}
         The initial values of the SDE system \eqref{eq:general model}, $X_0$ and $Y_0$, are $\mathcal{F}_0$-measureable and square integrable.
     \end{assumption}

     \begin{assumption}\label{as:jumps}
         The random measures $N_p$ and $N_\lambda$ do not jump at the same time almost surely. Let $D_p$ and $D_\lambda$ denote the set of jump times of $N_p([0,t],\mathbb{U}_p)$ and $N_\lambda([0,t],\mathbb{U}_\lambda)$ $\omega$ by $\omega$, then we have
         \begin{equation}
             \mathbb{P}(\{\omega\in\Omega:t\in D_p(\omega)\cap D_\lambda(\omega),\text{ for any } t\in[0,T]\})=0.
         \end{equation}
     \end{assumption}

     \begin{assumption}\label{as:invert sigma_2}
         The coefficient $\sigma_2:[0,T]\times\R^{d_Y}\rightarrow\R^{d_Y\times d_Y}$ takes values in the set of invertible matrices, and
         \begin{equation*}
             \sup_{t\in[0,T],y\in\R^{d_Y}}|\sigma_2^{-1}(t,x)|<\infty,
         \end{equation*}
         and the function $h:[0,T]\times\R^{d_X+d_Y}\rightarrow\R^{d_Y}$ defined as
         \begin{equation*}
             h(t,x,y) := \sigma_2^{-1}(t,y)\left(b_2(t,x,y) + \int_{\mathbb{U}_\lambda}f_2(t,y,u)(1-\lambda(t,x,u))\nu_2(du)\right),
         \end{equation*}
         satisfies
         \begin{equation*}
\sup_{(x,y)\in\R^{d_X+d_Y}}\int_0^T \vert h(t,x,y)\vert dt<\infty.
         \end{equation*}
     \end{assumption}

     \begin{assumption}\label{as:lambda integrability}
         The function $\lambda$ is uniformly bounded from above and below away from 0, and satisfies the integrability condition:
         \begin{equation}\label{eq:lambda integrability}
            \sup_{t\in[0,T],x\in\R^{d_X}} \int_{\mathbb{U}_\lambda}\frac{(1-\lambda(t,x,u))^2}{\lambda(t,x,u)}\nu_2(du)<\infty.
         \end{equation}
     \end{assumption}

    \begin{remark}
        The integrability condition \eqref{eq:lambda integrability} is fairly standard in stochastic filtering (see \cite[Assumption 5.5]{allan2025rough} and \cite[Assumption($\mathbf{H}_\lambda$)]{qiaocorr}). We note that Assumption \ref{as:lambda integrability} with $\lambda$ identically equal to 1 allows for infinite jumps as well (which we consider in Section \ref{sec:infinite}).
    \end{remark}
\subsection{Robustness of the Filter}\label{sec:Robust Filter}
    As observed in the introduction, given the signal-observation system~\eqref{eq:general model}, the problem of \textit{stochastic filtering} (see, e.g.~\citep{bain2009fundamentals,crisan2011oxford,kallianpur2013stochastic}) is to compute the conditional law of the signal $X_t$ given the filtration of the observation $\mathcal{F}^Y_t$ up to some time $t$. To be precise, given a measurable test function $f:\R^{d_X+d_Y}\rightarrow\R,$ the {\it filter} $\pi_t(f)$ is defined as the conditional expectation,
     \begin{equation*}
         \pi_t(f):=\mathbb{E}[f(X_t,Y_t)\vert\mathcal{F}^Y_t].
     \end{equation*}
   Under Assumptions \ref{as:lin boundedness}--\ref{as:lambda integrability}, for any bounded measureable function $f:\R^{d_X+d_Y}\rightarrow\R$, the Kallianpur-Striebel formula~\citep{crisan2011oxford} shows that
    \begin{equation}
        \pi_t(f) = \frac{\mathbb{\tilde{E}}[f(X_t,Y_t)\Lambda_t\vert\mathcal{F}_t^Y]}{\tilde{\mathbb{E}} [\Lambda_t\vert\mathcal{F}^Y_t]} =: \frac{p_t(f)}{p_t(1)}.
    \end{equation}

 In the following sections, we construct a version of the filter $\pi_t(f)$ which is continuous with respect to the observation process $Y$ under suitable topologies, the latter depending on the whether the common jump noise processes in~\eqref{eq:general model} have finite or infinite activity. This is known as a {\it robust filter}. The contours of our overall strategy for constructing a robust filter are to show that there exists a continuous version of the function $p_f(f)$, which we achieve by exploiting a measure change using Girsanov's theorem. Define a new reference measure $\tilde{\mathbb{P}}\ll \mathbb{P}$ using the Radon-Nikodym derivative:
    \begin{equation}
        \left.\frac{d\tilde{\mathbb{P}}}{d\mathbb{P}}\right\vert_{\mathcal{F}_T} := \Lambda_T^{-1},
    \end{equation}
    where the process $\Lambda_t := \exp(I_t)$, and
    \begin{multline*}
        I_t : =\int_0^th(s,X_s,Y_s)^TdW_s + \frac{1}{2}\int_0^t|h(s,X_s,Y_s)|^2ds + \int_0^t\int_{\mathbb{U}_\lambda}\log(\lambda(s,X_{s-},x))\tilde{N}_\lambda(ds,dx) \\ +\int_0^t\int_{\mathbb{U}_\lambda}(1-\lambda(s,X_{s-},x) + \lambda(s,X_{s-},x)\log(\lambda(s,X_{s-},x))\nu_2(dx)ds,
    \end{multline*}
    for $t\in[0,t]$ and the function $h$ is given in \eqref{as:invert sigma_2}. Note that in the infinite jump case, when $\lambda$ is identically 1, the last two terms in $I_t$ vanish and the transform is valid.

    Under the new measure $\tilde{\mathbb{P}}$, the signal-observation system \eqref{eq:general model} becomes:
    \begin{align}\label{eq:gris general model}
        dX_t &=  \tilde{b}_1(t,X_t,Y_t)dt + \sigma_0(t,X_t,Y_t)\circ dB_t + \sigma_1(t,X_t,Y_t)\circ d\tilde{W}_t \\ \nonumber &+ \int_{\mathbb{U}_p}f_1(t,X_{t-},Y_{t-},x)\tilde{N}_p(dt,dx) + \int_{\mathbb{U}_\lambda}f_3(t,X_{t-},Y_{t-},x)\tilde{N}(dt,dx), \\
        dY_t &:= \sigma_2(t,Y_t)\circ d\tilde{W}_t + \int_{\mathbb{U}_\lambda}f_2(t,Y_{t-},x)\tilde{N}(dt,dx).
    \end{align}

\section{Robustness for Scalar Systems}\label{sec:Flows}
 As a precursor to our main results, in this section we demonstrate the use of flows (such as in \eqref{eq:Flows of SDEs}) in the setting of jump-diffusion observation and signal models. The forthcoming results in this section clearly demonstrate the need for rough flows. These results closely follow \citep{davis1987pathwise}, who established the robustness in the one-dimensional setting, albeit with signal-observation diffusion models without jumps.

    To keep the exposition straightforward, we consider a simpler model (following \citep{qiaocorr}) which only has common Brownian noise, and the observation process is one-dimensional (in particular the common Brownian noise $W_t$ is also scalar):
    \begin{align}\label{eq:signal-obs-1d-corr-P}
    dX_t &= b_1(t,X_t)dt+\sigma_0(t,X_t)\circ dB_t + {\sigma}_1(W_t,X_t)\circ dW_t + \int_{|x|<1}f_1(t,X_{t-},x)\tilde{N}_p(dt,dx),\\
    dY_t &= \tilde{b}_2(W_t,X_t,Y_t)dt + {\sigma}_2(W_t,Y_t)\circ dW_t + \int_{|x|<1}f_2(t,Y_t,x)\tilde{N}_\lambda(dt,dx).
\end{align}

We assume that the SDEs have unique solutions (see \citep{qiaocorr}) and that $B_.,W_.,N_p,N_\lambda$ are mutually independent. We follow the notation used in the uncorrelated case.

As before, we use a Girsanov-transform to change $N_\lambda$ to unit intensity and eliminate the drift term in the observation process. To do this, define
\begin{multline}\label{eq:gis-mult-1d-corr}
    \Lambda_t^{-1} := \exp\left\{-\int_0^th(W_s,X_s,Y_s)\circ dW_s+\frac{1}{2}\int_0^t|h(W_s,X_s,Y_s)|^2ds \right. \\
    \left.- \int_0^t\int_{|x|<1}\log\lambda(s,X_{s-},x)N_\lambda(dt,dx) - \int_0^t\int_{|x|<1}(\lambda(s,X_{s-},x)-1)ds\nu_2(dx) \right\},
\end{multline}
where $h(W_t,X_t,Y_t) := \sigma^{-1}_2(W_t,Y_t)b_2(W_t,X_t,Y_t)$, and a new measure $\mathbb{\tilde{P}}$ via
\(
    \frac{d\mathbb{\tilde{P}}}{d\mathbb{P}} = \Lambda_t^{-1}.
\)
Under $\mathbb{\tilde{P}}$, we have
\begin{equation}
    \tilde{W}_t = W_t + \int_0^th(W_s,X_s,Y_s)ds
\end{equation}
is a Brownian motion,
\begin{equation}
    \tilde{N}(dt,dx) = N_\lambda(dt,dx)-dt\nu_2(dx)
\end{equation}
is a compensated Poisson random measure, and the signal-observation processes become
\begin{align}\label{eq:signal-obs-1d-corr}
    dX_t &= b_3(t,X_t,Y_t)dt+\sigma_0(t,X_t)\circ dB_t + \tilde{\sigma}_1(\tilde{W}_t,X_t)\circ d\tilde{W}_t + \int_{|x|<1}f_1(t,X_{t-},x)\tilde{N}_p(dt,dx),\\
    dY_t &= \tilde{\sigma}_2(\tilde{W}_t,Y_t)\circ d\tilde{W}_t + \int_{|x|<1}f_2(t,Y_t,x)\tilde{N}(dt,dx).
\end{align}
where
\begin{equation*}
    b_3(t,X_t,Y_t) = b_1(t,X_t)-\sigma_1(\tilde{W}_t,X_t)h(t,X_t,Y_t).
\end{equation*}
Furthermore, we also have that the Radon-Nikodym derivative $\Lambda_t$ is the solution of an SDE,
\begin{theorem}
    $\Lambda_t$ satisfies the following equation:
    \begin{equation}
        \Lambda_t = 1 + \int_0^t\Lambda_sh(W_s,X_s,Y_s)\circ d\tilde{W}_s+\int_0^t\int_{|x|<1}\Lambda_s(\lambda(s,X_s,x)-1)\tilde{N}(ds,dx).
    \end{equation}
\end{theorem}
\begin{proof}
    The proof follows that of Lemma 3.1 in \citep{qiaouncorr} and is omitted here.
\end{proof}

Of particular interest to our objective of establishing the robustness of nonlinear stochastic filters in the jump-diffusion setting, is the continuity of SDE flows. We defer the proofs of the following theorems,  which largely follow the results in~\citep{davis1987pathwise}, to the appendix.

\begin{theorem}\label{thm:sde-prop}
    Consider the SDE
    \begin{equation}\label{SDE flow}
        d\xi_t = Z(M_t,\xi_t)\circ dM_t,
    \end{equation}
    where $Z\in C_b^{\infty,\infty}$ are a class of commutative vector fields on $\mathbb{R}$ and $M_t$ is a continuous semi-martingale on $\mathbb{R}$ starting from 0.
    The SDE can be interpreted as:
    \begin{equation}\label{SDE func flow}
        f(\xi_t) = x + \int_0^t(Z(M_s)f)(\xi_t)\circ dM_s,
    \end{equation}
    where $f$ is any $C^\infty(\mathbb{R})$ function and $(Z(M_t)f)(x) := Z(M_t,x) \cdot \partial_x f(x)$ denotes the action of the vector field $Z(M_t, \cdot)$ at the point $x$ on the function $f$ evaluated at $x$.
    Then, the solution to the SDE started at $x$ is
    \begin{equation}\label{SDE soln}
        \xi_t = \exp \left(\int_0^{M_t}Z(s)ds \right)(x).
    \end{equation}
    where $\exp\left(\int_0^tZ(s)ds\right)(x)$ is the one parameter group of transformations defined as the solution to the ODE
    \begin{equation*}
        \frac{d}{dt}\phi_t(x) = Z(t,\phi_t(x))\hspace{25pt} \phi_0(x) = x,
    \end{equation*}
    and the solution is continuous in the sample paths of $M_t$.
\end{theorem}

Next, consider the SDE flow (as considered in \citep{davis1987pathwise})
\begin{equation}\label{eq:sde-flow-signal}
    \phi(t,x) = x+\int_0^t\sigma_1(\tilde{W}_t,\phi(s,x))\circ d\tilde{W}_s.
\end{equation}
The inverse flow is given by
\begin{equation}\label{eq:sde-inv-flow-signal}
    \psi(t,x) = \phi^{-1}(t,x) = x-\int_0^t\partial_x\psi(r,x)\sigma_1(\tilde{W}_r,x)\circ d\tilde{W_r}.
\end{equation}
We now consider the signal SDE \eqref{eq:signal-obs-1d-corr}, and compose it with the inverse flow given above. Using It\^o's theorem with Stratonovich integrals, we have
\begin{multline}\label{signal-inverse-flow}
    d\psi(t,X_t)= \partial_t\psi(t,X_t)dt + \partial_x\psi(t,X_t)\circ dX_t^c + \int_{|x|<1}[\psi(t,X_{t-}+f_1(t,X_t,x))-\psi(t,X_{t-})]\tilde{N}_p(dt,dx) \\+ \int_{|x|<1}[\psi(t,X_{t-}+f_1(t,X_t,x))-\psi(t,X_{t-})-f_1(t,X_t,x)\partial_x\psi(t,X_t)]dt\nu_1(dx).
\end{multline}
The first term can be expanded by differentiating \eqref{eq:sde-inv-flow-signal} with respect to $t$
\begin{align*}
    \partial_t\psi(t,X_t)dt = -\partial_x\psi(t,X_t)\sigma_1(\tilde{W}_t,X_t)\circ d\tilde{W}_t
\end{align*}
Plugging in $dX_t^c = b_3(t,X_t,Y_t)dt + \sigma_0(t,X_t)\circ dB_t + \sigma_t(\tilde{W_t},X_t)\circ d\tilde{W_t}$ and $\partial_t\psi(t,X_t)$, we get
\begin{align*}
    d\psi(t,X_t)=\partial_x\psi(t,X_t)b_3(t,X_t,Y_t)dt+\partial_x\psi(t,X_t)\sigma_0(t,X_t)\circ dB_t\\+\int_{|u|<1}[ \psi(t,X_{t-}+f_1(t,X_t,x))-\psi(t,X_{t-})]
    {N}_p(dt,du)-\int_{|u|<1}f_1(t,X_t,x)\partial_x\psi(t,X_t)dt\nu_1(du).
\end{align*}
Redefining coefficients
\begin{align}\label{eq:uncorr-signal}
    d\tilde{X}_t &:= d\psi(t,X_t)\\\nonumber &= \tilde{b}_3(t,\tilde{X}_t,Y_t)dt + \tilde{\sigma}_0(t,\tilde{X}_t)\circ dB_t + \int_{|x|<1}\tilde{f}_1(t,\tilde{X}_t,x)N_p(dt,dx) - \int_{|x|<1}\tilde{g}_1(t,\tilde{X}_t,x)dt\nu_1(dx),
\end{align}
where
\begin{align*}
    \tilde{b}_3(t,x,y) &= \partial_x\psi(t,x)b_3(t,x,y), \\
    \tilde{\sigma}_0(t,x) &= \partial_x\psi(t,x)\sigma_0(t,x), \\
    \tilde{f}_1(t,z,x) &= \psi(t,z+f_1(t,z,x))-\psi(t,z), \\
    \tilde{g}_1(t,z,x) &= \partial_x\psi(t,z)f_1(t,z,x).
\end{align*}
We now write the solution to equation \eqref{eq:signal-obs-1d-corr} as $X_t = \phi(t,\tilde{X}_t)$. We can now explicitly see that $X_t$ is continuous in sample paths of $\tilde{W}_t$. From stochastic flow theory, we also have that the flow $\phi$ is a diffeomorphism almost surely. Since, we need a Lipschitz condition on $b_2$ and $\sigma_2$ in equation \eqref{eq:signal-obs-1d-corr-P}, we have, $X_t$ is continuous in sample paths of $Y_t$.

\begin{theorem}\label{thm:continuity-sde-wrt-obs}
    Define the SDEs:
    \begin{align*}
        dX_t^1 &= b(t,X^1_t,y_t^1)dt+\sigma(t,X^1_t)\circ dB_t + \sigma_1(\tilde{w}^1_t,X^1_t)\circ d\tilde{w}^1_t\\ &\qquad\qquad\qquad + \int_{|x|<1}f(t,X^1_t,x){n}^1(dt,dx)+\int_{|x|<1}g_1(t,X^1_{t-},x)\tilde{N}_p(dt,dx)\\
        dX_t^2 &= b(t,X^2_t,y_t^2)dt+\sigma(t,X^2_t)\circ dB_t + \sigma_1(\tilde{w}^2_t,X^2_t)\circ d\tilde{w}^2_t\\ &\qquad\qquad\qquad + \int_{|x|<1}f(t,X^2_t,x){n}^2(dt,dx)+\int_{|x|<1}g_1(t,X^2_{t-},x)\tilde{N}_p(dt,dx)
    \end{align*}
    where $y_.^i$s are sample paths of $Y_.$ and $\tilde{w}_.^i$ and $n^i$ are the corresponding samples of $\tilde{W}_.$ and ${N}$ respectively. We also make the following assumptions:
    \begin{itemize}
        \item $b$ satisfies the inequality
        \begin{equation}
            |b(t,x_1,y_1)-b(t,x_2y_2)|\leq C_t|x_1-x_2\|y_1-y_2|.
        \end{equation}
        \item The vector fields $\sigma_1(\tilde{w}_t,.)$ are bounded, and Lipschitz.
        \item The function $g_1$ satisfies the following inequality
        \begin{equation}
            |g_1(t,z_1,x)-g_1(t,z_2,x)|\leq K_{t,z}|z_1-z_2|\alpha(x), \hspace{25pt}
        \end{equation}
        where $\alpha:[0,1)\rightarrow[0,1]$ is a function such that $\int_0^t\int_{|x|<1}\alpha(x)N_p(ds,dx)<\infty.$ a.s.
    \end{itemize}
    Then $X_t^i$s satisfy the continuity inequality:
    \begin{equation}
        \|X_t^1-X_t^2\|_{L_1} \lesssim \int_0^t\int_{|x|<1}|n^1(dr,dx)-n^2(dr,dx)| + |\tilde{w}^1_t-\tilde{w}^2_t|.
    \end{equation}
    \begin{equation}
        \|X\|_{L_1} := \mathbb{E}[|X|].
    \end{equation}
\end{theorem}

As a consequence of this theorem, $\Lambda_t$ and $X_t$ defined in equations \eqref{eq:gis-mult-1d-corr} and \eqref{eq:signal-obs-1d-corr}, respectively, are continuous in the sample paths of $Y_.$. Thus, if $f$ is locally Lipschitz and bounded, then we have the following estimate for two sample paths $y_.$ and $y'_.$ of $Y_.$ and the corresponding processes $X_t,X_t'$ and $\Lambda_t,\Lambda_t'$:
\begin{align*}
    f(X_t)\Lambda_t-f(X_t')\Lambda_t' &= f(X_t)\Lambda_t-f(X_t)\Lambda_t'+f(X_t)\Lambda_t'-f(X_t')\Lambda_t' \\
    &= f(X_t)(\Lambda_t-\Lambda'_t)+(f(X_t)-f(X_t'))\Lambda'_t \\
    &\leq|f|_\infty |\Lambda_t-\Lambda'_t|+|X_t-X_t'|\Lambda'_t,
\end{align*}
and we have,
\begin{equation}
    f(X_t)\Lambda_t-f(X_t')\Lambda_t'\rightarrow0 \hspace{10pt} \text{as} \hspace{10pt} |y_t-y'_t|\rightarrow0 \hspace{10pt}\forall \hspace{10pt}t\in[0,T]\hspace{20pt}.
\end{equation}
This proves the continuity of the filter along the sample paths of the observation, since from the Kallianpur-Striebel formula, we have
\begin{equation}
    \mathbb{E}[f(X_t)|\mathcal{Y}_t] = \frac{\mathbb{\tilde{E}}[f(X_t)\Lambda_t|\mathcal{Y}_t]}{\mathbb{\tilde{E}}[\Lambda_t|\mathcal{Y}_t]}.
\end{equation}
Thus, proving continuity of the numerator suffices (since the denominator is a special case with $f(x) = 1$).

\section{Robust Filtering with Finite Jumps}\label{sec:Finite}
 Coming to our primary contributions, we consider the signal-observation process $(X_t,Y_t)$ on $\mathbb{R}^{d_X}\times\mathbb{R}^{d_Y}$:
\begin{align}
dX_t &= b_1(t,X_t,Y_t)dt+\sigma_0(t,X_t,Y_t)\circ dB_t + {\sigma}_1(t,X_t,Y_t)\circ dW_t \nonumber \\
\label{eq:corr-signal-dims} &+ \int_{|x|<1}f_1(t,X_{t-},Y_t,x)\tilde{N}_p(dt,dx) + \int_{|x|<1}f_3(t,X_t,Y_t,x)\tilde{N}_\lambda(dt,dx),\\
dY_t &= {b}_2(t,X_t,Y_t)dt + {\sigma}_2(t,Y_t)\circ dW_t + \int_{|x|<1}f_2(t,Y_t,x)\tilde{N}_\lambda(dt,dx). \label{eq:corr-obs-dims}
\end{align}
In this section, the rough-path enhancement is required only for the continuous semimartingale component (the Brownian/observation-driven part), for which the geometric lift is naturally defined via the Stratonovich integral. The jump terms in (42)--(43) are treated in the classical It\^o--L\'evy sense via integrals against the (compensated) Poisson random measures and, under the standing assumption of finitely many jumps on compact time intervals, do not require a cadlag geometric rough-path lift of the jump driver. We assume that the SDEs have unique strong solutions (see \citep{qiaocorr}) and that $B_.,W_.,N_p,N_\lambda$ are mutually independent. The notation parallels that in the previous section. Since we now consider common jumps and multi-dimensional systems, we can no longer use the stochastic flow approach from the previous section, and instead we will use rough flows to establish the robustness result in considerable generality.

As before, we perform a Girsanov transform to change $N_\lambda$ to a unit intensity point process and eliminate the drift term in the observation process. To do this define
\begin{multline}\label{eq:rough-gris}
    \Lambda_t^{-1} = \exp\left\{-\int_0^th(s,X_s,Y_s)\circ dW_s+\frac{1}{2}\int_0^t|h(s,X_s,Y_s)|^2ds\right. \\
    \left.- \int_0^t\int_{|x|<1}\log\lambda(s,X_{s-},x)N_\lambda(dt,dx) - \int_0^t\int_{|x|<1}(\lambda(s,X_{s-},x)-1)ds\nu_2(dx)\right\},
\end{multline}
where $h(t,X_t,Y_t) = \sigma^{-1}_2(t,Y_t)b_2(t,X_t,Y_t)$.

Now define a new measure $\mathbb{\tilde{P}}$ via
 \(   \frac{d\mathbb{\tilde{P}}}{d\mathbb{P}} = \Lambda_t^{-1}.\)
Under $\mathbb{\tilde{P}}$, we have
\begin{equation}
    \tilde{W}_t = W_t + \int_0^th(s,X_s,Y_s)ds
\end{equation}
is a Brownian motion;
\begin{equation}
    \tilde{N}(dt,dx) = N_\lambda(dt,dx)-dt\nu_2(dx)
\end{equation}
is a compensated Poisson random measure, and the signal-observation processes become
\begin{align}
dX_t &= b_3(t,X_t,Y_t)dt+\sigma_0(t,X_t,Y_t)\circ dB_t + \sigma_1(t,X_t,Y_t)\circ d\tilde{W}_t \nonumber \\
\label{eq:1} &\qquad+ \int_{|x|<1}f_1(t,X_{t-},Y_t,x)\tilde{N}_p(dt,dx) + \int_{|x|<1}f_3(t,X_t,Y_t,x)\tilde{N}(dt,dx),\\
dY_t &= \sigma_2(t,Y_t)\circ d\tilde{W}_t + \int_{|x|<1}f_2(t,Y_t,x)\tilde{N}(dt,dx). \label{eq:2}
\end{align}
where
\begin{equation*}
    b_3(t,X_t,Y_t) = b_1(t,X_t,Y_t)-\sigma_1(t,X_t,Y_t)h(t,X_t,Y_t).
\end{equation*}
As before, our goal is to establish the existence of a robust version of the filtering function,
\begin{equation}\label{eq:kal-str-rough}
    \pi_t(f) := \mathbb{E}[f(X_t,Y_t)|\mathcal{Y}_t] = \frac{\tilde{\mathbb{E}}[f(X_t,Y_t)\Lambda_t|\mathcal{Y}_t]}{\tilde{\mathbb{E}}[\Lambda_t|\mathcal{Y}_t]}.
\end{equation}
To that end, consider the functional
\begin{equation}\label{eq:pt-rough}
    p_t(f) := \tilde{\mathbb{E}}[f(X_t,Y_t)\Lambda_t|\mathcal{Y}_t].
\end{equation}
Notice that both the numerator and the denominator of equation \eqref{eq:kal-str-rough} are of the form of equation \eqref{eq:pt-rough} (the denominator being the special case where $f(x) = 1$). Thus, if we show the continuity of a version of $p_t(f)$ with respect to the sample paths of the observation process $Y_\cdot$, we have our robust filter.
\begin{remark}
    Note that we do not need an assumption (needed in \citep{allan2025rough} and Section \ref{sec:infinite}) that the functions $f_2$ and $f_3$ can be split up as the products:
    \begin{align*}
        f_2(t,y,u) &= m_2(t,u)n_2(t,y),\\
        f_3(t,x,y,u) &= m_3(t,u)n_2(t,x,y).
    \end{align*}
    In the finite jump case, the jump discontinuities of the paths are of finite variation. Thus, we can treat the driving jump measures (integrated over all jump sizes) as paths of finite variation. This allows us to get away from stochastic integrals to get continuity bounds.
\end{remark}
\subsection{Continuous Rough Flows}
We require the following results on continuity of rough flows and rough SDEs (RSDEs):

\begin{theorem}\label{thm:rough-flow-cont}
    Consider the rough flow
    \begin{equation}\label{eq:rough-flow}
        \phi^{\rough{\eta}}(t,x) = x + \int_0^tc(\phi^{\rough{\eta}}(s,x))d\rough{\eta}_s,
    \end{equation}
    where $c\in \text{Lip}^{5+\gamma}$ for $\gamma>0$. We also assume that $\rough{\eta}_{p:[0,T]}<\infty$ is a geometric rough path . Then the rough flow $\phi^{\rough{\eta}}$ is continuous with respect to the driving rough path $\rough{\eta}$ and initial condition, when the space of rough paths is endowed with the rough $\alpha$-H\"older metric, where $\alpha = 1/p$.
\end{theorem}

\begin{proof}
    This is an elementary result in rough paths. See Theorem 2.12 part $(ii)$ in \citep{geng2021introduction}.
\end{proof}

As in the previous section, define the inverse flow (of equation \eqref{eq:rough-flow}) as (see Theorem \ref{thm:inv-flow-proof})
\begin{equation}\label{eq:inv-rough-flow}
    \psi^{\rough{\eta}}(t,x) := (\phi^{\rough{\eta}})^{-1}(t,x) := x-\int_0^t\partial_x\psi^{\rough{\eta}}(s,x)c(x)d{\rough{\eta}}_s.
\end{equation}

We now establish the existence of RSDEs with jumps. Here, we use geometric rough paths with finite jump activity and use rough flows to establish existence results.

\begin{theorem}
    Let $b\in C^1_b, \sigma_1\in C^1_b$, Consider $f_1,f_2:\mathcal{U}\times \mathbb{R}^{d_X}\rightarrow \mathbb{R}^{d\times d_N}$, where $\mathcal{U}$ is a subset of $\mathbb{R}^{d_N}$, that satisfy the following Lipschitz condition for each $x$:
    \begin{equation}
        |f_i(x,y_1)-f_i(x,y_2)|\leq C_{f_i}\alpha_i(x)|y_1-y_2|,\hspace{25pt} i=1,2.
    \end{equation}
    Consider a Poisson point process $P_t$ adapted to the probability space $(\Omega,\mathcal{F}, (\mathcal{F}_t)_{t\in[0,T]}, \mathbb{P})$ which takes values in $\mathcal{U}$, and $N_1$ its counting measure with compensator $\nu$. Define the compensated Poisson random measure $\tilde{N}_1(dt,dx):=N_1(dt,dx)-\nu_1(dx)dt$. Let $M_t$ be a standard Brownian motion on the same probability space. On an auxiliary probability space $(\bar{\Omega},\tilde{\mathcal{F}},\{\mathcal{\tilde{F}}_t\}_{0\leq t\leq T}, \mathbb{\tilde{P}})$, consider a Poisson random measure $N_2(dt,dx)$ with compensator $\nu_2(dx)dt$ such that $\int_{U_1}f(x)N_2([0,t],dx)$ is a path of finite variation for all $f\in C_b$, where $U_1\in\mathcal{U}$ is a set bounded from below such that $\int_{U_1}|\alpha_i(x)|^2\nu_i(dx)<\infty, i=1,2$. Further, let $\rough{\eta}\in\mathcal{C}^{0,\alpha}$ be a geometric rough path for $\alpha\in [1/2,1/3)$. Then the doubly stochastic RSDE:
    \begin{multline}\label{eq:double-stoch-RSDE}
        dX_t(\cdot,\bar{\omega}) := b(X_t(\cdot,\bar{\omega}))dt + \sigma_1(X_t(\cdot,\bar{\omega}))\circ dM_t + \sigma_2(X_t(\cdot,\bar{\omega}))d\rough{\eta}_t \\
        +\int_{U_1}f_1(x,X_t(\cdot,\bar{\omega}))\tilde{N}_1(dt,dx) + \int_{U_1}f_2(x,X_t(\cdot,\bar{\omega}))\tilde{N}_2(dt,dx)(\bar{\omega)},
    \end{multline}
    admits a solution.
\end{theorem}
\begin{proof}
    The proof is similar to the proof of \cite[Theorem 2]{crisanrough}, with the only difference being in the use of It\^o's lemma. In our case, we have continuous rough flows, but we are composing them with discontinuous processes (It\^o's lemma for discontinuous process of this form is given in \cite[Theorem 4.4.7]{applebaum2009levy}). However, this does not pose a problem since $\sigma_2\in C^5$, the flow and inverse flow are both twice differentiable and bounded and thus, Lipschitz. This guarantees the existence of a solution $\mathbb{\bar{P}}$ a.e. $\bar{\omega}$. 
\end{proof}

\begin{theorem}\label{thm:rough-poisson-cont}
    Let $b\in C^1_b, \sigma_{1}\in C^1_b,\sigma_2\in C^5_b$. Consider $f_1,f_2:\mathcal{U}\times \mathbb{R}^{d_X}\rightarrow \mathbb{R}^{d\times d_N}$, where $\mathcal{U}$ is a subset of $\mathbb{R}^{d_N}$, that satisfy the following Lipschitz condition for each $x$:
    \begin{equation}
        |f_i(x,y_1)-f_i(x,y_2)|\leq C_{f_i}\alpha_i(x)|y_1-y_2|,\hspace{25pt} i=1,2.
    \end{equation}
    Consider a Poisson point process $P_t$ adapted to the probability space $(\Omega,\mathcal{F}, (\mathcal{F}_t)_{t\in[0,T]}, \mathbb{P})$ which takes values in $\mathcal{U}$, and $N_1$ its counting measure with compensator $\nu$. Define the compensated Poisson random measure $\tilde{N}_1(dt,dx):=N_1(dt,dx)-\nu_1(dx)dt$. Let $M_t$ be a standard Brownian motion on the same probability space. \\
    Let $Q^1_t,Q^2_t$ be Poisson processes adapted to an auxiliary filtered probability space $(\bar{\Omega},\bar{\mathcal{F}},(\mathcal{\bar{F}}_t)_{0\leq t\leq T},\bar{\mathbb{P}})$ (we will later take this to be a copy of our original space) taking values in $\mathcal{U}$, and $N_q^1,N_q^2$ their counting measure with compensators $\nu_1$ and $\nu_2$ respectively. Let $X_t, \tilde{X}_t$ be continuous semi-martingales on $(\bar{\Omega},\bar{\mathcal{F}},(\mathcal{\bar{F}}_t)_{0\leq t\leq T},\bar{\mathbb{P}})$, and ${\rough{X}}(\bar{\omega}) := (X(\bar{\omega}),\mathbb{X}(\bar{\omega}))$ be the Stratonovich lift of $X_t$, i.e., $\mathbb{X}_{s,t} := \int_s^t(X_r-X_s)\otimes\circ dX_r$, and similarly ${\rough{\tilde{X}}}(\bar{\omega}) := (\tilde{X}(\bar{\omega}),\mathbb{\tilde{X}}(\bar{\omega}))$ the Stratonovich lift of $\tilde{X}_t$.\\
    Assume that $U_1\in\mathcal{U}$ bounded from below such that $\int_{U_1}|\alpha_i(x)|^2\nu_i(dx)<\infty$ for $i=1,2$.
   Then, the RSDEs
    \begin{multline}
        dY^{{\rough{X}}}_t(\cdot,\bar{\omega}) := b(Y^{{\rough{X}}}_t(\cdot,\bar{\omega}))dt+\sigma_1(Y^{{\rough{X}}}_t(\cdot,\bar{\omega}))\circ dM_t + \sigma_2(Y^{\rough{X}}_t(\cdot,\bar{\omega}))d{\rough{X}}_t(\bar{\omega})\\ + \int_{U_1}f_1(x,Y^{\rough{X}}_t(\cdot,\bar{\omega}))\tilde{N}_1(dt,dx) + \int_{U_1}f_2(x,Y^{\rough{X}}_t(\cdot,\bar{\omega}))\tilde{N}^1_q(dt,dx)(\bar{\omega}),
    \end{multline}
    \begin{multline}
        dY^{\rough{\tilde{X}}}_t(\cdot,\bar{\omega}) := b(Y^{\rough{\tilde{X}}}_t(\cdot,\bar{\omega}))dt+\sigma_1(Y^{\rough{\tilde{X}}}_t(\cdot,\bar{\omega}))\circ dM_t + \sigma_2(Y^{\rough{\tilde{X}}}_t(\cdot,\bar{\omega}))d{\rough{\tilde{X}}}_t(\bar{\omega})\\ + \int_{U_1}f_1(x,Y^{\rough{\tilde{X}}}_t(\cdot,\bar{\omega}))\tilde{N}_1(dt,dx) + \int_{U_1}f_2(x,Y^{\rough{\tilde{X}}}_t(\cdot,\bar{\omega}))\tilde{N}^2_q(dt,dx)(\bar{\omega}),
    \end{multline}
    defined on the product space $(\hat{\Omega},\hat{\mathcal{F}},\hat{\mathbb{P}}) := (\Omega\times\bar{\Omega},\mathcal{F}\otimes\bar{\mathcal{{F}}},\tilde{\mathbb{P}}\otimes\bar{\mathbb{P}})$, satisfy the following inequality for all $t\in[0,T]$ for $\bar{\mathbb{P}}$ almost every $\bar{\omega} \in \bar{\Omega}$
    \begin{align*}
\left\|Y^{\rough{X}}_t(\cdot,\bar{\omega})-Y^{\rough{\tilde{X}}}_t(\cdot,\bar{\omega})\right\|_{L^1(\Omega)}\leq &C(\bar{\omega})\left(\rho_{\alpha-\text{H\"ol}}\left({\rough{X}}(\bar{\omega}),{\rough{\tilde{X}}}(\bar{\omega}) \right) \right.\\&\qquad + \left.\left\|\int_{U_1}(\tilde{N}_q^1(ds,dx)(\bar{\omega})-\tilde{N}^2_q(ds,dx)(\bar{\omega}))\right\|_{1-\text{var}}\right),
    \end{align*}
    where $C(\bar{\omega})$ is $\bar{\mathbb{P}}$-almost-everywhere a finite constant.
\end{theorem}

\subsection{Continuity of the Filter}
Now, define processes $(X^{\rough{\eta},n},Y^{\rough{\eta},n},I^{\rough{\eta},n})$ as RSDEs on $(\bar{\Omega},\bar{\mathcal{F}},\bar{\mathbb{P}})$, where $\rough{\eta}$ is a rough path of finite p-variation with $p\in(2,3)$, and $\tilde{n}(dt,dx) = n(dt,dx)-dt\nu_2(dx)$ is some measure such that $\int_{|x|<1}f(x)n([0,t],dx)$ is a jump path with finite variation for all $f\in C_b$ and $t\in[0,T]$:
\begin{align}
    \begin{split}\label{eq:sig-obs-driv-rough}
        dX_t^{\rough{\eta},n} &:=  b_3(t,X_t^{\rough{\eta},n},Y_t^{\rough{\eta},n})dt+\sigma_0(t,X_t^{\rough{\eta},n},Y_t^{\rough{\eta},n})\circ dB_t + \sigma_1(t,X_t^{\rough{\eta},n},Y_t^{\rough{\eta},n})d\rough{\eta}_t\\&+ \int_{|x|<1}f_1(t,X_{t-}^{\rough{\eta},n},Y_t^{\rough{\eta},n},x)\tilde{N}_p(dt,dx) + \int_{|x|<1}f_3(t,X_t^{\rough{\eta},n},Y_t^{\rough{\eta},n},x)\tilde{n}(dt,dx),\\
        dY_t^{\rough{\eta},n} &:= \sigma_2(t,Y_t^{\rough{\eta},n})d\rough{\eta}_t + \int_{|x|<1}f_2(t,Y_t^{\rough{\eta},n},x)\tilde{n}(dt,dx),
    \end{split}
    \\[1ex]
    \begin{split}\label{eq:gris-driv-rough}
        dI_t^{\rough{\eta},n} &:= h(t,X_t^{\rough{\eta},n},Y_t^{\rough{\eta},n})d\rough{\eta}_t-\frac{1}{2}|h(t,X_t^{\rough{\eta},n},Y_t^{\rough{\eta},n})|^2dt \\
        &+ \int_{|x|<1}\log\lambda(t,X_{t-}^{\rough{\eta},n},x)\tilde{n}(dt,dx)+ \int_{|x|<1}(\lambda(t,X_{t-}^{\rough{\eta},n},x)-1+\log\lambda(t,X_{t-}^{\rough{\eta},n},x))ds\nu_2(dx).
    \end{split}
\end{align}

Also, define
\begin{equation}
    g^f_t(\rough{\eta},n) := \mathbb{\bar{E}}[f(X_t^{\rough{\eta},n},Y_t^{\rough{\eta},n})\exp\{I_t^{\rough{\eta},n}\}], \hspace{25pt} \theta_t^f(\rough{\eta},n):=\frac{g_t^f(\rough{\eta},n)}{g_t^1(\rough{\eta},n)}.
\end{equation}

\begin{remark}
    We do not have to worry about the boundedness of the exponential moment of $I^{\rough{\eta},n}$ since we are dealing with continuous RDEs (from flow decomposition). The boundedness was proved in \cite[Theorem 3]{crisanrough}
\end{remark}

Recall that $\mathcal{C}^\alpha$ is the space of rough paths with finite $1/\alpha$-variation, and $D$ is the space of paths with finite 1-variation.

\begin{theorem}
    Suppose $f\in C^1_b$, and assume the space $\mathcal{C}^\alpha([0,T]\bigoplus D([0,t]), \alpha\in(1/3,1/2)$ is endowed with the topology induced by the semi-norms $\rho_{\alpha-\text{H\"ol}} + \|\cdot\|_{1-\text{var}}$. Then $g^f_t(\rough{\eta},n)$ is Lipschitz.
\end{theorem}
\begin{proof}
    Consider two rough paths $\rough{\eta^1},\rough{\eta^2}\in\mathcal{C}^\alpha([0,T])$ and two measures $n^1,n^2$ such that for any $V\subseteq U_1$, $n_i(V)\in D([0,T]),i=1,2$. To shorten the notation, let the triple $(X,Y,I)$ driven by the pair $(\rough{\eta^i},n^i)$ be denoted $(X^i,Y^i,I^i)$. Consider the difference
    \begin{align*}
        |g^f_t(\rough{\eta^1},n^1)-g^f_t(\rough{\eta^2},n^2)| &= |\mathbb{\bar{E}}[f(X_t^1,Y_t^1)\exp\{I_t^1\}] - \mathbb{\bar{E}}[f(X_t^2,Y_t^2)\exp\{I_t^2\}]|\\
        &= |\mathbb{\bar{E}}[(f(X_t^1,Y_t^1)-f(X_t^2,Y_t^2))\exp\{I_t^1\}]\\&\qquad\qquad\qquad - \mathbb{\bar{E}}[f(X_t^2,Y_t^2) \left(\exp\{I_t^1\}-\exp\{I_t^2\}\right)]| \\
        &\leq C_f|\mathbb{\bar{E}}[(|X^1_t-X^2_t|+|Y_t^1-Y^2_t|)\exp\{I^1_t\}]|\\&\qquad\qquad\qquad + |f|_\infty\mathbb{\bar{E}}[(\exp\{I^1_t\}\vee\exp\{I^2_t\})|I^1_t-I^2_t|].
    \end{align*}
    For the last inequality we used the Lipschitz property of $f$ and the inequality $|\exp(x)-\exp(y)|\leq |x-y|(\exp(x)\vee\exp(y))$. We now use the boundedness of exponential moments of rough SDEs (proven in \cite[Theorem 3]{crisanrough}), along with Theorem \ref{thm:rough-poisson-cont}, we get
    \begin{equation*}
        |g^f_t(\rough{\eta^1},n^1)-g^f_t(\rough{\eta^2},n^2)|\leq C\left(\rho_{\alpha-\text{H\"ol}}(\rough{\eta^1},\rough{\eta^2}) + \left\|\int_{U_1}(n^1([0,\cdot],dx)-n^2([0,\cdot],dx))\right\|_{1-\text{var}:[0,t]}\right).
    \end{equation*}
\end{proof}

\subsection{Version of the Filter}

We now want to show that $\theta^f_t(\rough{\tilde{W}}(\bar{\omega}),N_\lambda(\bar{\omega}))$ is a version of the filter $\pi_t(f)(\bar{\omega})$ for $\mathbb{\bar{P}}$-almost-every $\bar{\omega}$. To do this, we need to ``lift" our signal-observation process to a product space in the following way.\\
Define an auxiliary probability space $(\bar{\Omega},\bar{\mathcal{F}},\bar{\mathbb{P}})$ on which the processes $B_t$ and $N_p$ lie. The processes $\tilde{W}_t$ and $N_\lambda$ lie in the space $(\Omega,\mathcal{F},\mathbb{\tilde{P}})$. The signal-observation processes (\eqref{eq:1},\eqref{eq:2}) as well as the Grisanov multiplier (\eqref{eq:rough-gris}) lie on the product space $(\hat{\Omega},\hat{\mathcal{F}},\hat{\mathbb{P}}) := (\Omega\times\bar{\Omega},\mathcal{F}\otimes\bar{\mathcal{{F}}},\tilde{\mathbb{P}}\otimes\bar{\mathbb{P}})$. Note that $(X,Y,I)$ have the same distribution on $(\hat{\Omega},\hat{\mathcal{F}})$ under $\hat{\mathbb{P}}$ as it did on $(\Omega,\mathcal{F})$ under $\mathbb{\tilde{P}}$.

\begin{theorem}\label{thm:consistency}
    Let $\rough{\tilde{W}}_t(\bar{\omega}) = (\tilde{W}_t(\bar{\omega}), \mathbb{\tilde{W}}_t(\bar{\omega}))$ be the Stratonovich rough lift of $\tilde{W}_t(\bar{\omega})$. Then, we have
    \begin{equation}
        g^f_t(\rough{\tilde{W}},N_\lambda) = \rho_t(f),\hspace{25pt} \theta_t^f(\rough{\tilde{W}},N_\lambda) = \pi_t(f)
    \end{equation}
    $\mathbb{\tilde{P}}$ almost surely.
\end{theorem}
\begin{proof}
    First notice that for all $A\in \mathcal{F}^Y_t$, we have
    \begin{equation*}
        \mathbb{\tilde{E}}[f(X_t,Y_t)\exp\{I_t\}\mathbf{1}_A]=\mathbb{\hat{E}}[f(X_t,Y_t)\exp\{I_t\}\mathbf{1}_A] = \mathbb{\tilde{E}}[\mathbb{\bar{E}}[f(X_t,Y_t)\exp\{I_t\}]\mathbf{1}_A]
    \end{equation*}
    Thus, the random variable $\mathbb{\bar{E}}[f(X_t,Y_t)\exp\{I_t\}]$ is $\mathcal{F}^Y_t$ measurable and is a version of the conditional expectation $\mathbb{\tilde{E}}[f(X_t,Y_t)\exp\{I_t\}|\mathcal{F}^Y_t]$.
    It is a classical result in rough path theory (Theorem \ref{thm:rde=sde}) that $(X_t^{\rough{\tilde{W}},N_\lambda}, Y_t^{\rough{\tilde{W}},N_\lambda},I_t^{\rough{\tilde{W}},N_\lambda}) = (X_t,Y_t,I_t)$, $\mathbb{\tilde{P}}$ almost surely. Thus, we have $\mathbb{\tilde{P}}$ almost surely
    \begin{equation*}
        \mathbb{\bar{E}}[f(X_t,Y_t)\exp\{I_t\}] = \mathbb{\bar{E}}[f(X_t^{\rough{\tilde{W}},N_\lambda},Y_t^{\rough{\tilde{W}},N_\lambda})\exp\{I_t^{\rough{\tilde{W}},N_\lambda}\}].
    \end{equation*}
\end{proof}

\section{Robust Filtering with Infinite Jumps}\label{sec:infinite}
Generalized shot noise models can also have countably infinite number of jumps in a finite time interval. Consider a shot noise model of the form
    \[
    d\xi_t = \int_{|x|< 1}f(t,x)\tilde{N}(dt,dx),
    \]
    where $\mathbb{E}[N(dt,dx)] = \nu(dx)dt$ and $\int_{|x|<1}|x|\nu(dx)$ is not finite. We also assume that the function $f$ is bounded, and the measure $\nu$ satisfies the following property
    \begin{align}\label{eq:p var compensator}
        \int_{|x|<1}|x|^p\nu(dx) &<\infty, p\in[2,3).
    \end{align}
    This ensures that the shot noise $\xi_t$ has finite $p$-variation for $p\in[2,3)$. Nonlinear stochastic filtering problems with an infinite intensity shot noise cannot be treated with the frame work that we have developed thus far. As in Section \ref{sec:Finite}, one may separate the jump component into "large" and "small" jumps. For any fixed cutoff $\varepsilon>0$, the contribution of jumps with $|z|>\varepsilon$ is finite on compact intervals and can be treated as discrete events exactly as in the finite-activity framework. The essential new feature in the present setting is the infinite activity of the small-jump component $|z|\le \varepsilon$, whose cumulative effect requires a cadlag geometric enhancement. We therefore employ the canonical/Marcus construction (equivalently, a log-linear fill-in of jumps) to obtain a geometric rough-path lift that is stable under $\varepsilon\downarrow 0$.

    Before going into the model and deriving robustness of the filtering function, we observe that the solution map for the canonical rough differential equation (RDE) in~\eqref{eq:canonical-RDE} is continuous.

    \begin{theorem}\label{thm:lip flow of canonical RDE}
        The solution map for~\eqref{eq:canonical-RDE} is continuous (and locally Lipschitz) as a map
        \begin{equation}
\left(\mathbb{R}^e,|\cdot|\right)\oplus\left(\mathcal{D}^{p-\textnormal{var}}(G^N(\mathbb{R}^d),\beta_p)\right)\mapsto \left(D^{p-\textnormal{var}}(\mathbb{R}^e),\|\|_{p-\textnormal{var}}\right),
        \end{equation}
        and the associated flow is a diffeomorphism $\Phi\in \textnormal{\text{Diff}}^m(\mathbb{R}^e)$.
    \end{theorem}

    \begin{remark}
        Since we have shown continuity with respect to the $p$-variation metric (which is a pseudo-metric), when we have $\|y^1-y^2\|_{p-\text{var}} = 0$, we can conclude that $y^1 = y^2+c$, i.e., the paths differ by a constant. However, noticing that the RHS of the inequality
        \begin{equation*}
            \|\tilde{y}^1-\tilde{y}^2\|_{p-\text{var}} \lesssim |y^1_0-y_0^2| + \rho_{p-\text{var}}(\mathbf{x}^{\phi^1,\delta}_1,\mathbf{x}^{\phi^2,\delta}_2).
        \end{equation*}
        has a $|y^1_0-y^2_0|$ which also needs to go to zero, the constant $c$ must be zero, and we have $y^1\rightarrow y^2$ as $y^1_0\rightarrow y^2_0$ and $x_1\rightarrow x_2$ in $\beta_{p-\text{var}}$.
    \end{remark}

    We now state and prove a slightly more general result about the stability of rough SDEs (RSDEs), where the rough path is cadlag and of finite $p$-variation.
    \begin{corollary}\label{thm:RSDE cadlag stability}
        Let $p\in[2,3)$ and $(\Omega,\mathcal{F},\mathbb{P})$ be a probability space carrying a standard Brownian motion $B_t$ and a semimartingale $\xi_t$ that has cadlag sample paths of finite $p$-variation almost surely. Let $({\rough{\eta}}^1,\phi^1),({\rough{\eta}}^2,\phi^2)\in \mathcal{D}^{p-\text{var}}$ be two representatives, and consider the two RSDEs:
        \begin{align}\label{eq:canonical RSDE stability}
            dX^1_t(\omega)&:= b(X_t^1(\omega))dt+\sigma(X_t^1(\omega))\circ dB_t(\omega) + f(X_t^1(\omega))\diamond d\xi_t + g(X_t^1)\diamond d({\rough{\eta}}^1,\phi^1),\\
            dX^2_t(\omega)&:= b(X_t^2(\omega))dt+\sigma(X_t^2(\omega))\circ dB_t(\omega) + f(X_t^2(\omega))\diamond d\xi_t + g(X_t^2)\diamond d({\rough{\eta}}^2,\phi^2),
        \end{align}
        Then we have the following continuity estimate:
        \begin{equation}
            \|X^1_\cdot-X^2_\cdot\|_{p-\text{var};[0,t]}\lesssim |X^1_0-X^2_0|+\beta_p(\rough{\eta}^1,\rough{\eta}^2).
        \end{equation}
    \end{corollary}

\subsection{The Model}
Consider the signal-observation process $(X_t,Y_t)$ on $\mathbb{R}^{d_X}\times\mathbb{R}^{d_Y}$:
\begin{align}
dX_t &= b_1(t,X_t,Y_t)dt+\sigma_0(t,X_t,Y_t)\circ dB_t + {\sigma}_1(t,X_t,Y_t)\circ dW_t \nonumber \\
\label{eq:corr-signal-inf} &\qquad\qquad\qquad+ f_1(t,X_{t-},Y_t)\diamond d\xi_t^1 +f_3(t,X_t,Y_t)\diamond d\xi^2_t,\\
dY_t &= {b}_2(t,X_t,Y_t)dt + {\sigma}_2(t,Y_t)\circ dW_t + f_2(t,Y_t)\diamond d\xi_t^2, \label{eq:corr-obs-inf}
\end{align}
where
\begin{equation*}
    d\xi^i_t = \int_{|x|<1}\eta^i(t,x)\tilde{N}^i(dt,dx),\hspace{10pt} \int_{|x|<1}\eta^i(t,x)\nu_i(dx)<\infty,\hspace{10pt}\forall t\in[0,T], i=1,2.\
\end{equation*}

We assume that the SDEs have unique strong solutions (see \citep{qiaocorr}) and that $B_.,W_.,N_p,N_\lambda$ are mutually independent. The notation used is the same as that in the previous section. The difference now is that both ${N}^1$ and ${N}^2$ have compensators that satisfy equation \eqref{eq:p var compensator}. The compensated measure of $N_\lambda(dt,dx)$ is $dt\nu_2(dx)$ (instead of $\lambda(x,X_{t-},t)\nu_2(dx)$). The reason we consider this is that if the compensator is $X$ dependent, then it is not easy to remove it since the $\Lambda_t^{-1}$ defined in \eqref{eq:rough-gris} by the new $\tilde{N}_\lambda$ is no longer a martingale and the Girsanov theorem does not hold.

We perform a change of measure as usual
\begin{equation}
    \Lambda^{-1}_t = \exp\left\{-\int_0^th(s,X_s,Y_s)\circ dW_s+\frac{1}{2}\int_0^t|h(s,X_s,Y_s)|^2ds\right\},
\end{equation}
where $h(t,X_t,Y_t) = (\sigma_2(t,Y_t))^{-1}b_2(t,X_t,Y_t)$.
Define $\tilde{\mathbb{P}}\ll \mathbb{P}$ as
\begin{equation}
    \frac{d\tilde{\mathbb{P}}}{d\mathbb{P}} = \Lambda^{-1}_t.
\end{equation}
Under $\tilde{\mathbb{P}}$, the signal-observation system \eqref{eq:corr-signal-inf},\eqref{eq:corr-obs-inf} becomes
\begin{align}
    dX_t &= b_3(t,X_t,Y_t)dt+ \sigma_0(t,X_t,Y_t)\circ dB_t+\sigma_1(t,X_t,Y_t)\circ d\tilde{W}_t \nonumber \\
    &\qquad\qquad\qquad+f_1(t,X_{t-},Y_t)\diamond d\xi_t^1 + f_3(t,X_t,Y_t)\diamond d\xi^2_t,\\
    dY_t &= {\sigma}_2(t,Y_t)\circ d\tilde{W}_t +f_2(t,Y_t)\diamond d\xi_t^2,
\end{align}
where $\tilde{W}_t = W_t - h(t,X_t,Y_t)$ is a standard Brownian motion under $\tilde{\mathbb{P}}$. Following the Kallianpur-Striebel formula, we have
\begin{equation}
    \pi_t(f) := \mathbb{E}[f(X_t,Y_t)|\mathcal{Y}_t] = \frac{\mathbb{\tilde{E}}[f(X_t,Y_t)\Lambda_t|\mathcal{Y}_t]}{\mathbb{\tilde{E}}[\Lambda_t|\mathcal{Y}_t]} = \frac{p_t(f)}{p_t(1)}.
\end{equation}

Define the triple $(X^{\rough{\eta}},Y^{\rough{\eta}},I^{\rough{\eta}})$ as RSDEs on $(\bar{\Omega},\bar{\mathcal{F}},\mathbb{\bar{P}})$, where $\rough{\eta}$ is a cadlag geometric rough path of finite $p$-variation with $p\in[2,3)$.
\begin{align}\label{eq:triple cadlag}
    dX^{\rough{\eta}}_t &:= b_3(t,X^{\rough{\eta}}_t,Y^{\rough{\eta}}_t)dt + \sigma_0(t,X^{\rough{\eta}}_t,Y^{\rough{\eta}}_t)\circ dB_t + f_1(t,X^{\rough{\eta}}_t,Y^{\rough{\eta}}_t)\diamond d\xi_t^1 + G_1(t,X^{\rough{\eta}}_t,Y^{\rough{\eta}}_t)\diamond d\rough{\eta}_t,\\ \nonumber
    dY^{\rough{\eta}}_t &:= G_2(t,Y^{\rough{\eta}}_t)\diamond d\rough{\eta}_t, \\ \nonumber
    dI^{\rough{\eta}}_t &:= \frac{1}{2}h(t,X^{\rough{\eta}}_t,Y^{\rough{\eta}}_t)dt+G_3(t,X^{\rough{\eta}}_t,Y^{\rough{\eta}}_t)\diamond d\rough{\eta_t},
\end{align}
where,
\begin{align*}
G_1(t,X^{\rough{\eta}}_t,Y^{\rough{\eta}}_t) &:= (\sigma_1(t,X^{\rough{\eta}}_t,Y^{\rough{\eta}}_t),f_3(t,X^{\rough{\eta}}_t,Y^{\rough{\eta}}_t)),\\
    G_2(t,,Y^{\rough{\eta}}_t) &:= (\sigma_2(t,Y^{\rough{\eta}}_t),f_2(t,Y^{\rough{\eta}}_t)), \\
    G_3(t,X^{\rough{\eta}}_t,Y^{\rough{\eta}}_t)&:= (h(t,X^{\rough{\eta}}_t,Y^{\rough{\eta}}_t),0).
\end{align*}
Corollary \ref{thm:RSDE cadlag stability} implies that the RSDEs are continuous in $p$-variation topology with respect to the rough path $\rough{\eta}$ (in the topology generated by $\beta_p$).

Now consider the process $L:= (\tilde{W},\xi^2)^T$. The discontinuous part of $L$ is only from the $\xi^2$ component of the vector, and thus the Marcus integral of the first component of $L$ reduces to a Stratonovich integral. Rewriting the SDEs in terms of $G_1,G_2$ and $G_3$, we get
\begin{align}\label{eq:rewritten sig-obs}
    dX_t &= b_3(t,X_t,Y_t)dt +\sigma_0(t,X_t,Y_t)\circ dB_t + f_1(t,X_t,Y_t)\diamond d\xi_t^1 + G_1(t,X_t,Y_t)\diamond dL_t,\\
    \nonumber
    dY_t &= G_2(t,Y_t)\diamond dL_t,\\
    \nonumber
    d\Lambda_t &= \exp\left(-\frac{1}{2}h(t,X_t,Y_t)dt+G_3(t,X_t,Y_t)\diamond dL_t\right).
\end{align}

Let $\mathbf L$ represent the Marcus lift of $L$. We now state the result which is a consequence of previous theorems.
\begin{theorem}\label{thm:RSDE=SDE}
    The RSDEs in equations \eqref{eq:triple cadlag} are continuous as a map from $(\mathcal{D}^{p-\text{var}},\beta_p)\rightarrow(D^{p-\text{var}}(\mathbb{R}^e),\|\cdot\|_{p-\text{var}})$ and locally Lipschitz if $(G_1,G_2,G_3)$ are $\text{Lip}^{\gamma+m-1}$ with $\gamma>p$ and $m>2$. The RSDE solutions are also a version of the SDE solutions from equations \eqref{eq:rewritten sig-obs} when $\rough{\eta} = \mathbf L$.
\end{theorem}

\begin{proof}
    The continuity is a consequence of Theorem \ref{thm:lip flow of canonical RDE}. We only need to prove that the RSDEs with $\rough{\eta} = \rough{L}$ is a version of the signal-observation processes. First, consider the general SDE
    \begin{equation*}
        dK_t :=\alpha(K_t)dt+\beta(K_t)\diamond dM_t + \gamma(K_t)\diamond dL_t,
    \end{equation*}
    where $M_t$ is a semi-martingale, and the RSDE,
    \begin{equation*}
        dK_t^r := \alpha(K_t^r)dt + \beta(K_t^r)\diamond dM_t + \gamma(K_t^r)\diamond d\rough{L}_t.
    \end{equation*}
    From \cite[Proposition 4.16 ]{chevyrev2019canonical}, we know that the flows of the SDE and RSDE (resp.)
    \begin{align*}
        \phi(t,x) &:= x + \int_0^t\gamma(\phi(s,x))\diamond dL_t,\\
        \phi^r(t,x) &:= x + \int_0^t\gamma(\phi^r(s,x))\diamond d\rough{L}_t,
    \end{align*}
    coincide almost surely, implying their inverses also coincide (almost surely) and are diffeomorphisms. Consider the composition of the inverse flows with the processes, i.e.,
    \begin{align*}
        d\tilde{K}_t :=d\phi^{-1}(t,K_t)&= \tilde{\alpha}(\tilde{K}_t)dt+\tilde{\beta}(\tilde{K}_t)\diamond dM_t,\\
        d\tilde{K}^r_t := d(\phi^r)^{-1}(t,K_t^r) &= \tilde{\alpha}^r(\tilde{K}^r_t)dt + \tilde{\beta}^r(\tilde{K}^r_t)\diamond dM_t.
    \end{align*}
    Since $\tilde{\alpha}^r = \tilde{\alpha}$ and $\tilde{\beta}^r = \tilde{\beta}$ almost surely, from the uniqueness of the solution to a strong SDE, we have $\tilde{K}^r_\cdot = \tilde{K}_\cdot$, and thus $K^r_\cdot = K_\cdot$, almost surely.

    The RSDEs and SDEs of interest to us are special cases, and are therefore we have $(X,Y,\Lambda) = (X^{\rough{\eta}},Y^{\rough{\eta}},\exp(I^{\rough{\eta}}))$ almost surely.
\end{proof}

In order to show that we have a version of the filter, define an auxiliary probability space $(\bar{\Omega},\bar{\mathcal{F}},\bar{\mathbb{P}})$ which is a copy of our original space and carries the processes $\xi_t^1$ and $B_t$. Define the functionals:
\begin{align*}
    g_t^f(\rough{\eta)}&:= \bar{\mathbb{E}}[f(X^{\rough{\eta}}_t,Y^{\rough{\eta}}_t)\exp(I^{\rough{\eta}})],\\
    \theta^f_t(\rough{\eta}) &:= \frac{g^f_t(\rough{\eta})}{g_t^1(\rough{\eta})}.
\end{align*}
\begin{theorem}
 We have
 \(       g_t^f(\rough{L)} = \rho_t(f)\)  and \(\theta^f_t(\rough{L}) = \pi_t(f)\),
    $\tilde{\mathbb{P}}$ almost surely.
\end{theorem}

\begin{proof}
    The proof follows that of Theorem \ref{thm:consistency} and Theorem \ref{thm:RSDE=SDE}.
\end{proof}

\section{Conclusions and Future Work}
We have established the robustness of nonlinear stochastic filters for signal-observation models driven by L\'evy-type noise. By exploiting Stratonovich and Marcus formulations and the associated flow decompositions, we constructed versions of the filter that depend continuously on the observation path in two regimes: finite jump activity (Section~\ref{sec:Finite}) and infinite jump activity (Section~\ref{sec:infinite}). In the finite-activity case, we obtained continuity in both the rough
$p$-variation and
$p$-variation topologies on cadlag path space, without requiring the separability assumption on the jump coefficients imposed in~\citep{allan2025rough}. In the infinite-activity case, we proved continuity in a modified rough
$p$-variation topology under an analogous separability condition.

A central feature of our approach is that the Stratonovich and Marcus lifts yield geometric rough paths, which brings several structural advantages over the It\^o-based framework of~\citep{allan2025rough}. First, the pathwise nature of the Stratonovich and Marcus integrals provides convergence guarantees that do not depend on knowledge of the underlying probability law, whereas It\^o integrals converge only in probability. Second, the resulting flows are diffeomorphisms, making our framework readily extensible to signal-observation models on manifolds, a setting of considerable importance in constrained engineering systems where It\^o calculus is awkward to apply. Third, the geometric rough-path signatures satisfy the shuffle product identity and Chen's relation, opening a path toward efficient numerical computation of the filter via signature methods.

Several directions remain open. The separability condition on the jump coefficients $f_2$
 and $f_3$ is still required in the infinite-activity regime (Section~\ref{sec:infinite}), and relaxing or removing it is an important challenge. A related question is whether the Girsanov transform can be extended to handle state-dependent compensators in the infinite-activity setting, which would bring the model in Section~\ref{sec:infinite} closer to the generality of Section~\ref{sec:Finite}. On the computational side, exploiting the signature structure of the geometric lift to develop practical filtering algorithms, particularly for high-dimensional or real-time applications, is a natural next step. Finally, extending the robustness results to signal-observation models evolving on Riemannian manifolds, taking full advantage of the diffeomorphic flow structure established here, is a promising direction for future work.

\bibliographystyle{plainnat}
\bibliography{ref}

@article{friz2012,
 ISSN = {00911798},
 author = {Peter K. Friz and Atul Shekhar},
 journal = {The Annals of Probability},
 number = {4},
 pages = {2707--2765},
 publisher = {Institute of Mathematical Statistics},
 title = {General rough integration, L\'evy rough paths and a L\'evy–Khintchine-type formula},
 urldate = {2026-02-14},
 volume = {45},
 year = {2017}
}

@book{applebaum2009levy,
  title={{L{\'e}vy Processes and Stochastic Calculus}},
  author={Applebaum, David},
  year={2009},
  publisher={Cambridge University Press}
}

@article{qiaocorr,
  title={{Convergence of nonlinear filtering for multiscale systems with correlated L{\'e}vy noises}},
  author={Qiao, Huijie},
  journal={Stochastics and Dynamics},
  volume={23},
  number={02},
  pages={2350016},
  year={2023},
  publisher={World Scientific}
}

@article{clark,
  title={On a robust version of the integral representation formula of nonlinear filtering},
  author={Clark, J. M. C.  and Crisan, Dan},
  journal={Probability Theory and Related Fields},
  volume={133},
  number={1},
  pages={43--56},
  year={2005},
  publisher={Springer}
}

@article{crisanrough,
author = {D. Crisan and J. Diehl and P. K. Friz and H. Oberhauser},
 journal = {The Annals of Applied Probability},
 number = {5},
 pages = {2139--2160},
 publisher = {Institute of Mathematical Statistics},
 title = {Robust Filtering: Correlated Noise and Multidimensional Observation},
 urldate = {2026-02-14},
 volume = {23},
 year = {2013}
}

@inproceedings{kunita2006decomposition,
  title={On the decomposition of solutions of stochastic differential equations},
  author={Kunita, Hiroshi},
  booktitle={Stochastic Integrals: Proceedings of the LMS Durham Symposium, July 7--17, 1980},
  pages={213--255},
  year={2006},
  organization={Springer}
}

@article{qiaouncorr, title={Nonlinear filtering of stochastic dynamical systems with {L\'evy} noises}, volume={47}, DOI={10.1239/aap/1444308887}, number={3}, journal={Advances in Applied Probability}, author={Qiao, Huijie and Duan, Jinqiao}, year={2015}, pages={902–918}}

@article{davis1987pathwise,
  title={Pathwise nonlinear filtering for nondegenerate diffusions with noise correlation},
  author={Davis, {Mark H.A.} and Spathopoulos, {Michael P.}},
  journal={SIAM Journal on Control and Optimization},
  volume={25},
  number={2},
  pages={260--278},
  year={1987},
  publisher={SIAM}
}

@article{geng2021introduction,
  title={{An Introduction to the Theory of Rough Paths}},
  author={Geng, Xi},
  journal={Lecture Notes},
  pages={9},
  year={2021}
}

@article{allan2025rough,
  title={Rough {SDEs} and Robust Filtering for Jump-Diffusions},
  author={Allan, Andrew L and Pieper, Jost and Teichmann, Josef},
  journal={arXiv preprint arXiv:2507.05930},
  year={2025}
}

@book{friz2014course,
  title={{A Course on Rough Paths}},
  author={Friz, Peter K and Hairer, Martin},
  year={2014},
  publisher={Springer}
}

@article{chevyrev2019canonical,
  title={Canonical {RDEs} and general semimartingales as rough paths},
  author={Chevyrev, Ilya and Friz, Peter K},
  journal={The Annals of Probability},
  volume={47},
  number={1},
  pages={420--463},
  year={2019},
  publisher={JSTOR}
}

@article{kurtz1995stratonovich,
  title={Stratonovich stochastic differential equations driven by general semimartingales},
  author={Kurtz, Thomas G and Pardoux, {\'E}tienne and Protter, Philip},
  journal={Annales de l'IHP Probabilit{\'e}s et Statistiques},
  volume={31},
  number={2},
  pages={351--377},
  year={1995}
}

@article{gubinelli2010ramification,
  title={Ramification of rough paths},
  author={Gubinelli, Massimiliano},
  journal={Journal of Differential Equations},
  volume={248},
  number={4},
  pages={693--721},
  year={2010},
  publisher={Elsevier}
}

@incollection{chevyrev2016primer,
author="Chevyrev, Ilya
and Kormilitzin, Andrey",
editor="Bayer, Christian
and dos Reis, Goncalo
and Horvath, Blanka
and Oberhauser, Harald",
title="A Primer on the Signature Method in Machine Learning",
booktitle="Signature Methods in Finance: An Introduction with Computational Applications",
year="2026",
publisher="Springer Nature Switzerland",
pages="3--64"
}

@article{marcus2003modeling,
  title={Modeling and analysis of stochastic differential equations driven by point processes},
  author={Marcus, Steven I},
  journal={IEEE Transactions on Information Theory},
  volume={24},
  number={2},
  pages={164--172},
  year={2003},
  publisher={IEEE}
}

@article{marcus1981modeling,
  title={Modeling and approximation of stochastic differential equations driven by semimartingales},
  author={Marcus, Steven I},
  journal={Stochastics: An International Journal of Probability and Stochastic Processes},
  volume={4},
  number={3},
  pages={223--245},
  year={1981},
  publisher={Taylor \& Francis}
}

@book{friz2010multidimensional,
  title={{Multidimensional Stochastic Processes as Rough Paths: Theory and Applications}},
  author={Friz, Peter K and Victoir, Nicolas B},
  volume={120},
  year={2010},
  publisher={Cambridge University Press}
}

@book{bain2009fundamentals,
  title={{Fundamentals of Stochastic Filtering}},
  author={Bain, Alan and Crisan, Dan},
  volume={3},
  year={2009},
  publisher={Springer}
}

@book{kallianpur2013stochastic,
  title={{Stochastic Filtering Theory}},
  author={Kallianpur, Gopinath},
  volume={13},
  year={2013},
  publisher={Springer Science \& Business Media}
}

@book{crisan2011oxford,
  title={{The Oxford Handbook of Nonlinear Filtering}},
  author={Crisan, Dan and Rozovskii, Boris},
  year={2011},
  publisher={Oxford University Press}
}

@article{clark1978design,
  title={The design of robust approximations to the stochastic differential equations of nonlinear filtering},
  author={Clark, J. M. C. },
  journal={Communication Systems and Random Process Theory},
  volume={25},
  pages={721--734},
  year={1978},
  publisher={Sijthoff and Noordhoff Alphen aan den Rijn}
}

@article{elliott1981robust,
  title={Robust filtering for correlated multidimensional observations},
  author={Elliott, Robert J and Kohlmann, Michael},
  journal={Mathematische Zeitschrift},
  volume={178},
  number={4},
  pages={559--578},
  year={1981},
  publisher={Springer}
}

\begin{appendix}
\section{}\label{app:proofs}
\subsection{Proofs of Theorems in Section \ref{sec:Flows}}
\begin{proof}[Proof of Theorem~\ref{thm:sde-prop}]
    We first prove the equivalence of the two SDEs \eqref{SDE flow} and \eqref{SDE func flow}. Going from \eqref{SDE func flow} to \eqref{SDE flow} is trivial (putting $f(\xi) = \xi$). We first recall the action of a vector field on a function. Given a vector field $\mathcal X$ and a function $f\in C^\infty$, the action of the vector field on the function $f$ at a point $x$ is the gradient of $f$ in the direction of $\mathcal X$. Specifically,
    \begin{equation*}
        (\mathcal Xf)(x) = \mathcal X(x) \cdot \partial_x f(x)
    \end{equation*}
    Now consider the SDE~\eqref{SDE flow}. Using It\^o's lemma for Stratonovich integrals,
    \begin{align*}
        df(\xi_t) &= \partial_x f(\xi_t)\circ d\xi_t \\
        &= \partial_x f(\xi_t) \cdot Z(M_t,\xi_t)\circ dM_t \\
        &= (Z(M_t)f)(\xi_t)\circ dM_t.
    \end{align*}

    Consider the function $F(t,x) := f\left(\exp\left(\int_0^tZ(s)\,ds\right)(x)\right)$,  for some $f\in C^\infty(\mathbb{R})$. Then, we have $d F(t,x) = Z\left(t,f\left(\exp\left(\int_0^tZ({s})ds \right)(x) \right) \right) \cdot dt$. To see this, using the notation used in the theorem $\phi_t(x) := \exp\left(\int_0^tZ(s)ds\right)(x)$ and
    \begin{align*}
        \frac{d}{dt}f(\phi_t(x)) &= \partial_x f(\phi_t(x))\times \frac{\partial}{\partial t}\phi_t(x) \\
        &= \partial_x f(\phi_t(x)) \times Z(t,\phi_t(x)) \\
        &= Z(t,\phi_t(x))\times \partial_x f(\phi_t(x)) \\
        &= (Z(t)f)(\phi_t(x)),
    \end{align*}
    where we use the definition of $\exp\left(\int_0^t Z(s) \,ds\right)(x)$ as the solution to the ODE in the second equality and the action of the vector field on $f$ at the point $\exp\left(\int_0^tZ_s\,ds \right)(x)$ in the last equality, thereby implying that $d F(t,x) = Z\left(t,f\left(\exp\left(\int_0^tZ({s})ds \right)(x) \right) \right) \cdot dt$.

    Now, using It\^o's formula, we have
    \begin{equation*}
        dF(M_t,x) = \frac{\partial}{\partial t}F(t,x) \bigg\vert_{t= M_t} =  Z\left(M_t,f\left(\exp\left(\int_0^{M_t}Z(s)\,ds \right)(x) \right)\right)\circ dM_t, \hspace{25pt} F(M_0,x) = x
    \end{equation*}
    Thus, from the uniqueness of the strong solution to an SDE, the solution of~\eqref{SDE flow} is
    \begin{equation}
        \xi_t = \exp\left(\int_0^{M_t}Z(s) \,ds \right)(x).
    \end{equation}

    In order to show that the solution to the SDE is continuous in sample paths of the martingale $M_t$, we consider the flow $(\exp{\int_0^tZ(s,x)ds)}(x)$ of the ODE
    \begin{equation*}
        \frac{d}{dt}\phi_t(x) = Z(t,\phi_t(x)),\hspace{25pt} \phi_0(x) = x.
    \end{equation*}
    Consider the difference of flows at two times $t$ and $t'$,
    \begin{align*}
        \left|(\exp\{\int_0^tZ(s,x)ds)\}(x)\right. & \left.-(\exp\{\int_0^{t'}Z(s,x)ds)\}(x) \right|\\ &= \left|(x+\int_0^tZ(s,\phi_s(x))ds)-(x+\int_0^{t'}Z(s,\phi_s(x))ds) \right| \\
        &= \left|\int_{t'}^tZ(s,\phi_s(x))ds \right|\\
        &\leq |Z|_\infty |t-t'|,
    \end{align*}
    where $|Z|_\infty = \sup_{t\in \mathbb{R},x\in\mathbb{R}}Z(t,x)$.\\
    We know that the solution is obtained by composing the flow $(\exp{\int_0^.Z(s,x)ds)}(x):\mathbb{R}\rightarrow\mathbb{R}$ with the map $M_.:\mathbb{R}_+\rightarrow\mathbb{R}$. Thus, for two sample paths $m_t$ and $m'_t$,
    \begin{align*}
        |(\exp{\int_0^{m_t}Z(s,x)ds)}(x)-(\exp{\int_0^{m'_t}Z(s,x)ds)}(x)|\leq |Z|_\infty|m_t-m'_t|.
    \end{align*}
\end{proof}

\begin{proof}[Proof of Theorem \ref{thm:continuity-sde-wrt-obs}]
    Similar to equation \eqref{signal-inverse-flow} define $\tilde{X}_t^y$ as
    \begin{multline*}
        d\tilde{X}_t^y= \tilde{b}(t,\tilde{X}^y_t,y_t)dt + \tilde{\sigma}(t,\tilde{X}^y_t)\circ dB_t + \int_{|x|<1}\tilde{f}(t,x)n(dt,dx) - \int_{|x|<1}\tilde{g}(t,\tilde{X}^y_t,x)dt\nu_1(dx) \\+ \int_{|x|<1}\tilde{g}_1(t,\tilde{X}_{t-}^y,x)\tilde{N}_p(dt,dx),
    \end{multline*}
    where,
    \begin{align*}
        \tilde{b}(t,x,y) &= \partial_x\psi(t,x)b(t,x,y), \\
        \tilde{\sigma}(t,x) &= \partial_x\psi(t,x)\sigma(t,x), \\
        \tilde{g}_1(t,z,x) &= \psi(t,z+g_1(t,z,x))-\psi(t,z), \\
        \tilde{g}(t,z,x) &= \partial_x\psi(t,z)g(t,z,x).
    \end{align*}
    and consider the flow
    \begin{equation*}
        \psi^{-1}(t,x) = \phi(t,x) = x + \int_0^t\sigma(s,\phi(s,x))\circ d\tilde{w}_s.
    \end{equation*}
    Then we have $\phi(t,\tilde{X}_t^y) = X_t^y$. We have for two sample paths $y^1$ and $y^2$ the two flows
    \begin{align}
        \phi^1(t,x) = x + \int_0^t\sigma(\tilde{w}^1_s,\phi^1(s,x))\circ d\tilde{w}^1_s,\\
        \phi^2(t,x) = x + \int_0^t\sigma(\tilde{w}^2_s,\phi^2(s,x))\circ d\tilde{w}_s^2.
    \end{align}
    Consider the difference of the flows,
    \begin{align*}
        \left|\phi^1(t,x)-\phi^2(t,x) \right| &= \left|\int_0^t\sigma(\tilde{w}^1_s,\phi^1(s,x))\circ d\tilde{w}^1_s-\int_0^t\sigma(\tilde{w}^2_s,\phi^2(s,x))\circ d\tilde{w}^2_s \right| \\
        &\leq |\sigma|_\infty \left|\tilde{w}_t^1+\tilde{w}_t^2\right|,
    \end{align*}
    where $|\sigma|_\infty = \sup_{t\in \mathbb{R},x\in\mathbb{R}}\sigma(t,x)$.
    Since $\sigma$ is bounded and $\tilde{w}_t^i$ is finite a.s., the difference between the two flows is finite. Also notice that the RHS has no $x$ dependence. Thus, we have
    \begin{equation*}
        \sup_{x\in\mathbb{R}} \left|\phi^i(t,x) \right| \leq |\sigma|_\infty \left|\tilde{w}^i_t\right|<\
        \infty.
    \end{equation*}
    We can also show Lipschitz continuity of the flows since the vector fields are $C^\infty(\mathbb{R})$, we know that the flows are also $C^\infty(\mathbb{R})$ and thus locally Lipschitz.
    The difference between the signal process driven by two paths of the observation is:
    \begin{align*}
        \left|X_t^1-X_t^2 \right|_{L_1} &= \left|\phi^1(t,\tilde{X}_t^1) - \phi^2(t,\tilde{X}_t^2) \right|_{L_1}\\
        &\leq \left|\phi^1(t,\tilde{X}_t^1) - \phi^1(t,\tilde{X}_t^2) \right|_{L_1}+ \left|\phi^1(t,\tilde{X}_t^2) - \phi^2(t,\tilde{X}_t^2) \right|_{L_1} \\
        & \leq \left|\phi^1(t,\tilde{X}_t^1) - \phi^1(t,\tilde{X}_t^2)\right|_{L_1} + \sup_{x\in\mathbb{R}} \left|\phi^1(t,x)-\phi^2(t,x) \right|.
     \end{align*}
    The last term is not an $L_2$ norm since it has no random variables, thus, only the modulus remains.
    For the maximum process we have
    \begin{equation*}
        \sup_{s\leq t} \left|(X_s^1-X_s^2) \right|_{L_1} \leq \sup_{s\leq t} \left|\phi^1(s,\tilde{X}_s^1) - \phi^1(s,\tilde{X}_s^2) \right|_{L_1} + \sup_{x\in\mathbb{R},s\leq t} \left|\phi^1(s,x)-\phi^2(s,x) \right|
    \end{equation*}

    Since we know that the flows are continuous in both the initial point and in the driving martingale, we have that the second term is continuous in paths of $Y_.$. We only need to show continuity of the first term. Since the flow is $C^\infty$ in the initial condition it is locally Lipschitz. Thus,
    \begin{equation*}
        \left|\phi^1(t,\tilde{X}_t^1) - \phi^2(t,\tilde{X}_t^2) \right|_{L_1}\leq K_{1,2,t} \left|\tilde{X}_t^1-\tilde{X}_t^2 \right|_{L_1}
    \end{equation*}
    To complete the proof, write the processes $\left|\tilde{X}_t^1-\tilde{X}_t^2\right|$,
    \begin{align*}
        &\sup_{r\leq t} \left|\tilde{X}_r^1-\tilde{X}_r^2 \right|\\
        &\quad\leq \sup_{r\leq t} \bigg\{\int_0^rC_s \left|\tilde{X}_s^1-\tilde{X}_s^2 \right| \left|y^1_s-y^2_s \right|ds + \int_0^rU \left|\tilde{X}_s^1-\tilde{X}_s^2 \right|\circ dB_s \\
        &\qquad+\int_0^r\int_{|x|<1}V\alpha(x) \left|\tilde{X}_s^1-\tilde{X}_s^2 \right|\tilde{N}_p(ds,dx)+\int_0^r\int_{|x|<1}\tilde{f}(s,x) \left|n^1(ds,dx)-n^2(ds,dx) \right| \bigg\}.
    \end{align*}
    Here, we used the Lipschitz properties of functions (see \cite[Theorem 6.2.3]{applebaum2009levy} and \cite[Lemma 9]{crisanrough}), and $U$ and $V$ are the Lipschitz constants.
    Taking an expectation on both sides, we have
    \begin{align*}
         &\mathbb{E}[\sup_{r\leq t}|\tilde{X}_r^1-\tilde{X}_r^2|]\\ &\quad\leq \mathbb{E} \left[\sup_{r\leq t} \left\{\int_0^rC_s \left|\tilde{X}_s^1-\tilde{X}_s^2 \right| \left|y^1_s-y^2_s \right|ds + \int_0^rU \left|\tilde{X}_s^1-\tilde{X}_s^2 \right|\circ dB_s  \right.\right. \\ &\qquad\left. \left.+\int_0^r\int_{|x|<1}V\,\alpha(x) \left|\tilde{X}_s^1-\tilde{X}_s^2 \right|\tilde{N}_p(ds,dx)]+\sup_{r\leq t}\int_0^r\int_{|x|<1}\tilde{f}(s,x) \left|n^1(ds,dx)-n^2(ds,dx) \right| \right\} \right].
    \end{align*}
    We can now use the BDG inequality on the expectation on the RHS to get
    \begin{align*}
        \mathbb{E}\left[\sup_{r\leq t} \left|\tilde{X}_r^1-\tilde{X}_r^2 \right| \right]\leq \mathbb{E}\left[\left(\int_0^tU^2 \left|\tilde{X}_s^1-\tilde{X}_s^2 \right|^2ds+\int_0^r\int_{|x|<1}V^2\,\alpha(x)^2 \left|\tilde{X}_s^1-\tilde{X}_s^2 \right|^2ds\nu_2(dx) \right)^{1/2} \right.\\ \left.+ \sup_{r\leq t}\int_0^r\int_{|x|<1}\tilde{f}(s,x) \left|n^1(ds,dx)-n^2(ds,dx) \right| \right].
    \end{align*}
    Taking the maximum of the process $|\tilde{X}_s^1-\tilde{X}_s^2|^2$ on the RHS, we get $\mathbb{E} \left[\sup_{r\leq t} \left|\tilde{X}_r^1-\tilde{X}_r^2 \right| \right] \leq $
    \begin{multline*}
         \mathbb{E} \left[ \left(\sup_{s\leq t} \left|\tilde{X}_s^1-\tilde{X}_s^2 \right|^2\int_0^tU^2\,ds + \sup_{s\leq t} \left|\tilde{X}_s^1-\tilde{X}_s^2 \right|^2\int_0^r\int_{|x|<1}V^2\,\alpha(x)^2ds\nu_2(dx) \right)^{1/2} \right. \\ \left.+ \sup_{r\leq t}\int_0^r\int_{|x|<1}\tilde{f}(s,x) \left|n^1(ds,dx)-n^2(ds,dx) \right| \right].
    \end{multline*}
    Taking the maximum process common and noticing that the remaining terms are deterministic, we have
    \begin{align*}
        \mathbb{E} \left[\sup_{r\leq t} \left|\tilde{X}_r^1-\tilde{X}_r^2 \right| \right]\leq \mathbb{E} \left[ \left(\sup_{s\leq t} \left|\tilde{X}_s^1-\tilde{X}_s^2 \right|^2 \right)^{1/2} \right]&\left(\int_0^tU^2 \,ds +\int_{|x|<1}\int_0^tV^2\,\alpha(x)^2\,ds\nu_s(dx) \right)^{1/2} \\ &+ \sup_{r\leq t}\int_0^r\int_{|x|<1}\tilde{f}(s,x) \left|n^1(ds,dx)-n^2(ds,dx) \right|.
    \end{align*}
    Since $x^{1/2}$ is a monotonically increasing function, we have $(\sup_{s\leq t}X_s)^{1/2} = \sup_{s\leq t}X_s^{1/2}$. Thus,
    \begin{align*}
        \mathbb{E} \left[\sup_{r\leq t} \left|\tilde{X}_r^1-\tilde{X}_r^2 \right| \right]\leq \mathbb{E} \left[\sup_{s\leq t} \left|\tilde{X}_s^1-\tilde{X}_s^2 \right| \right] & \left(\int_0^tU^2ds+\int_{|x|<1}\int_0^tV^2\,\alpha(x)^2ds\nu_s(dx) \right)^{1/2}\\ &+ \sup_{r\leq t}\int_0^r\int_{|x|<1}\tilde{f}(s,x) \left|n^1(ds,dx)-n^2(ds,dx) \right|.
    \end{align*}
    Applying Gronwall's inequality to get
    \begin{multline*}
        \mathbb{E} \left[\sup_{r\leq t} \left|\tilde{X}_r^1-\tilde{X}_r^2 \right| \right]\leq \sup_{r\leq t} \left\{\int_0^r\int_{|x|<1}\tilde{f}(s,x) \left|n^1(ds,dx)-n^2(ds,dx) \right| \right\}\\\times \exp \left( \left(\int_0^tU^2 \, ds + \int_{|x|<1}\int_0^t V^2\alpha(x)^2 \,ds\,\nu_s(dx) \right)^{1/2} \right).
    \end{multline*}
    Using the fact that $\tilde{f}$ is bounded, we have $\mathbb{E}
        \left[\sup_{r\leq t}
        \left|\tilde{X}_r^1-\tilde{X}_r^2 \right| \right]\leq $
    \begin{equation*}
        \left|\tilde{f} \right|_\infty\sup_{r\leq t} \int_0^r\int_{|x|<1}\left|n^1(dt,dx)-n^2(dt,dx) \right|\exp \left( \left(\int_0^tU^2 ds+\int_{|x|<1}\int_0^tV^2\alpha(x)^2ds\,\nu_s(dx) \right)^{1/2} \right).
    \end{equation*}
    Thus, putting everything together, we have
    \begin{equation*}
        \left|\sup_{s\leq t} \left(X_s^1-X_s^2 \right) \right|_{L_1}\leq \tilde{K}_{t}\, \sup_{r\leq t} \int_0^r\int_{|x|<1}\left|n^1(dt,dx)-n^2(dt,dx) \right| + L_t\, \left|w_t^1-w_t^2\right|,
    \end{equation*}
    where $\tilde{K}_{t} = \left|\tilde{f} \right|_\infty \exp \left( \left(\int_0^tU^2 \, ds+\int_{|x|<1}\int_0^tV^2\alpha(x)^2 \, ds\,\nu_s(dx) \right)^{1/2} \right)$.
\end{proof}

\subsection{Proofs of Theorems in Section \ref{sec:Finite}}

\begin{proof}[Proof of Theorem \ref{thm:rough-poisson-cont}]
    Consider the rough flows
    \begin{align*}
        \phi^1(t,x) &:= x + \int_0^t\sigma_2(\phi^1(s,x))d{\rough{X}}_s(\bar{\omega}), \\
        \phi^2(t,x) &:= x + \int_0^t\sigma_2(\phi^2(s,x))d{\rough{\tilde{X}}}_s(\bar{\omega}),
    \end{align*}
    and their corresponding inverse flows $\psi^1,\psi^2$ as defined in equation \eqref{eq:inv-rough-flow}. From It\^o's Lemma,
    \begin{multline*}
        dY^1_t(\cdot,\bar{\omega}) := d\psi^1(t,Y^{\rough{X}}_t(\cdot,\bar{\omega})) = \tilde{b}(Y^1_t)dt+\tilde{\sigma}_1(Y^1_t)\circ dM_t \\+ \int_{U_1}\tilde{f}_1(x,Y^1_t)\tilde{N}_1(dt,dx) + \int_{U_1}\tilde{f}_2(x,Y^1_t)\tilde{N}^1_q(dt,dx)(\bar{\omega}),
    \end{multline*}
    \begin{multline*}
        dY^2_t(\cdot,\bar{\omega}) := d\psi^2(t,Y^{\rough{\tilde{X}}}_t(\cdot,\bar{\omega})) = \tilde{b}(Y^2_t)dt+\tilde{\sigma}_1(Y^2_t)\circ dM_t \\+ \int_{U_1}\tilde{f}_1(x,Y^2_t)\tilde{N}_1(dt,dx) + \int_{U_1}\tilde{f}_2(x,Y^2_t)\tilde{N}^2_q(dt,dx)(\bar{\omega}).
    \end{multline*}
    By Theorem \ref{thm:rough-flow-cont},
    \begin{equation*}
        |\phi^1(t,x_1)-\phi^2(t,x_2)|\le C\bigl(|x_1-x_2|+\rho_{\alpha-\text{H\"ol}}({\rough{X}}(\bar{\omega}),{\rough{\tilde{X}}}(\bar{\omega}))\bigr),
    \end{equation*}
    so
    \begin{multline}\label{eq:rough-cont-ineq}
        \left\|Y^{\rough{X}}_t(\cdot,\bar{\omega})-Y^{\rough{\tilde{X}}}_t(\cdot,\bar{\omega})\right\|_{L^1(\Omega)} = \left\|\phi^1(t,Y^1_t)-\phi^2(t,Y^2_t)\right\|_{L^1(\Omega)} \\
        \le C\left(\left\|Y^1_t(\cdot,\bar{\omega})-Y^2_t(\cdot,\bar{\omega})\right\|_{L^1(\Omega)} + \rho_{\alpha-\text{H\"ol}}({\rough{X}}(\bar{\omega}),{\rough{\tilde{X}}}(\bar{\omega}))\right).
    \end{multline}
    It remains to bound $\|Y^1_t-Y^2_t\|_{L^1(\Omega)}$. We work in $L^2(\Omega)$ to preserve the Gronwall form and pass to $L^1$ by Jensen's inequality at the end. Define the non--decreasing process
    \begin{equation*}
        \Phi(t):=\hat{\mathbb{E}}\!\left[\sup_{s\le t}|Y^1_s(\cdot,\bar\omega)-Y^2_s(\cdot,\bar\omega)|^2\right].
    \end{equation*}
    Subtracting the two SDEs and adding and subtracting $\int_0^r\!\!\int_{U_1}\tilde f_2(x,Y^1_s)\tilde N^2_q(ds,dx)(\bar\omega)$, we write $Y^1_r-Y^2_r=\sum_{i=1}^5 A^i_r$ where
    \begin{align*}
        A^1_r&:=\int_0^r[\tilde b(Y^1_s)-\tilde b(Y^2_s)]\,ds, & A^2_r&:=\int_0^r[\tilde\sigma_1(Y^1_s)-\tilde\sigma_1(Y^2_s)]\circ dM_s,\\
        A^3_r&:=\int_0^r\!\!\int_{U_1}[\tilde f_1(x,Y^1_s)-\tilde f_1(x,Y^2_s)]\tilde N_1(ds,dx), & A^4_r&:=\int_0^r\!\!\int_{U_1}[\tilde f_2(x,Y^1_s)-\tilde f_2(x,Y^2_s)]\tilde N^2_q(ds,dx)(\bar\omega),\\
        A^5_r&:=\int_0^r\!\!\int_{U_1}\tilde f_2(x,Y^1_s)\bigl[\tilde N^1_q-\tilde N^2_q\bigr](ds,dx)(\bar\omega). &&
    \end{align*}
    From $(a+b+c+d+e)^2\le 5(a^2+b^2+c^2+d^2+e^2)$, taking $\sup_{r\le t}$ and $\hat{\mathbb E}$,
    \begin{equation}\label{eq:5split}
        \Phi(t)\le 5\sum_{i=1}^{5}\hat{\mathbb{E}}\!\left[\sup_{r\le t}(A^i_r)^2\right].
    \end{equation}
    We bound each term so that $\Phi(t)$ appears only inside a Stieltjes integral, never multiplicatively.

    \smallskip\noindent\emph{(i) Drift.} By Cauchy--Schwarz and the Lipschitz property of $\tilde b$, $\sup_{r\le t}(A^1_r)^2\le t\,C_b^2\int_0^t|Y^1_s-Y^2_s|^2\,ds$, hence
    \begin{equation*}
        \hat{\mathbb{E}}[\sup_{r\le t}(A^1_r)^2]\le t\,C_b^2\int_0^t\Phi(s)\,ds.
    \end{equation*}

    \smallskip\noindent\emph{(ii) Brownian Stratonovich integral.} Convert Stratonovich to It\^o (the correction is a bounded--variation term absorbable into (i)), then apply Doob's $L^2$--maximal inequality and the It\^o isometry:
    \begin{equation*}
        \hat{\mathbb{E}}[\sup_{r\le t}(A^2_r)^2]\le 4C_\sigma^2\int_0^t\Phi(s)\,ds.
    \end{equation*}

    \smallskip\noindent\emph{(iii) Compensated Poisson $\tilde N_1$.} By Doob's $L^2$--maximal inequality and the $L^2$ isometry for compensated Poisson integrals,
    \begin{equation*}
        \hat{\mathbb{E}}[\sup_{r\le t}(A^3_r)^2]\le 4C_{f_1}^2\!\int_{U_1}\!\alpha_1(x)^2\,\nu_1(dx)\int_0^t\Phi(s)\,ds.
    \end{equation*}

    \smallskip\noindent\emph{(iv) Conditional pure--jump integral.} For fixed $\bar\omega$, $\tilde N^2_q(\cdot,\bar\omega)$ is a signed measure of finite total variation on $[0,T]\times U_1$ (finite--jump setting). Let $|\tilde N^2_q|$ denote its total variation, and set
    \begin{equation*}
        d\mu_2(s,\bar\omega):=\int_{U_1}\alpha_2(x)\,|\tilde N^2_q|(ds,dx)(\bar\omega),\qquad m_2(t,\bar\omega):=\mu_2([0,t],\bar\omega)<\infty\ \text{a.s.}
    \end{equation*}
    By the Lipschitz property of $\tilde f_2$ and Cauchy--Schwarz for the Stieltjes integral,
    \begin{equation*}
        (A^4_r)^2\le C_{f_2}^2\Bigl(\int_0^r|Y^1_s-Y^2_s|\,d\mu_2(s,\bar\omega)\Bigr)^{\!2}\le C_{f_2}^2\,m_2(t,\bar\omega)\int_0^t|Y^1_s-Y^2_s|^2\,d\mu_2(s,\bar\omega).
    \end{equation*}
    Since $\mu_2(\cdot,\bar\omega)$ is $\bar\omega$--measurable and independent of $\omega$, Fubini yields
    \begin{equation*}
        \hat{\mathbb{E}}[\sup_{r\le t}(A^4_r)^2]\le C_{f_2}^2\,m_2(t,\bar\omega)\int_0^t\Phi(s)\,d\mu_2(s,\bar\omega).
    \end{equation*}

    \smallskip\noindent\emph{(v) Source term.} Using $|\tilde f_2|\le|\tilde f_2|_\infty$,
    \begin{equation*}
        \sup_{r\le t}|A^5_r|\le|\tilde f_2|_\infty\,V_q(t,\bar\omega),\qquad V_q(t,\bar\omega):=\left\|\int_{U_1}\bigl(\tilde N^1_q-\tilde N^2_q\bigr)(ds,dx)(\bar\omega)\right\|_{1\text{-var};[0,t]}.
    \end{equation*}

    \smallskip\noindent\emph{Combining.} Define the positive finite--variation control
    \begin{equation*}
        d\Lambda(s,\bar\omega):=\bigl(tC_b^2+4C_\sigma^2+4C_{f_1}^2\|\alpha_1\|^2_{L^2(\nu_1)}\bigr)\,ds+C_{f_2}^2\,m_2(t,\bar\omega)\,d\mu_2(s,\bar\omega),
    \end{equation*}
    which satisfies $\Lambda([0,t],\bar\omega)<\infty$ a.s. Inserting (i)--(v) into \eqref{eq:5split},
    \begin{equation}\label{eq:gronwall-ready}
        \Phi(t)\le 5\int_0^t\Phi(s)\,d\Lambda(s,\bar\omega)+5|\tilde f_2|_\infty^2\,V_q(t,\bar\omega)^2.
    \end{equation}
    This is in Gronwall form. By Theorem \ref{thm:gronwall} (with $\Lambda$ as the finite--variation integrator in place of $|dx_s|$),
    \begin{equation*}
        \Phi(t)\le 5|\tilde f_2|_\infty^2\,V_q(t,\bar\omega)^2\exp\!\bigl(5\Lambda([0,t],\bar\omega)\bigr).
    \end{equation*}
    Jensen's inequality, $\hat{\mathbb{E}}[|Y^1_t-Y^2_t|]\le\Phi(t)^{1/2}$, then gives
    \begin{equation*}
        \bigl\|Y^1_t-Y^2_t\bigr\|_{L^1(\Omega)}\le\tilde C(\bar\omega)\,V_q(t,\bar\omega),\qquad \tilde C(\bar\omega):=\sqrt{5}\,|\tilde f_2|_\infty\exp\!\bigl(\tfrac{5}{2}\Lambda([0,t],\bar\omega)\bigr)<\infty\ \bar{\mathbb P}\text{-a.s.}
    \end{equation*}
    Substituting into \eqref{eq:rough-cont-ineq} yields the claimed bound.
\end{proof}

\subsection{Proofs of Theorems in Section \ref{sec:infinite}}
\begin{proof}[Proof of Theorem \ref{thm:lip flow of canonical RDE}]
        Construct canonical RDEs
        \begin{align*}
            dy^1_t &= V(y^1_t)\diamond d(\mathbf{x}^1_t,\phi^1),\\
            dy^2_t &= V(y^2_t)\diamond d(\mathbf{x}^2_t,\phi^2).
        \end{align*}
        For the continuous RDE solutions $\tilde{y}^1_t$ and $\tilde{y}^2_t$, we have the standard stability estimates (see Theorem 3.13 in \citep{chevyrev2019canonical})
        \begin{equation*}
            \|\tilde{y}^1-\tilde{y}^2\|_{p-\text{var}} \lesssim |y^1_0-y_0^2| + \rho_{p-\text{var}}(\mathbf{x}^{\phi^1,\delta}_1,\mathbf{x}^{\phi^2,\delta}_2).
        \end{equation*}
        Now recall the map $\tau$ in~\eqref{eq:time-ext} and $\tau_r(t) = t(T+r)$/T, and consider the effect of $\tau_x = \tau_r^{-1}\circ \tau$ on the $p$ variation of a path. Since $\tau_r$ and $\tau_r^{-1}$ are reparameterizations of the path, the $p$ variation does not change.

        Next, for any path $x\in C^{p-\text{var}}([0,T],E)$ on a Banach space E, the path $x\circ \tau_r$ is a path in $C^{p-\text{var}}([0,T+r],E)$ with
        \begin{equation*}
            \|x\|_{p-\text{var};[0,T]} = \|x\circ\tau_r\|_{p-\text{var};[0,T+r]}.
        \end{equation*}
        Let $\hat{x} = x\circ\tau_r$, then the path $\hat{x}\circ\tau$ has the property,
        \begin{equation*}
            \|\hat{x}\circ\tau\|_{p-\text{var};[0,T+r]}\leq \|\hat{x}\|_{p-\text{var};[0,T+r]} = \|x\|_{p-\text{var}[0,T]},
        \end{equation*}
        since
        \begin{align*}
            \|\hat{x}\|^p_{p-\text{var};[0,T+r]} &= \sup_{\mathcal{P}_{[0,T+r]}}\left(\sum_{u,v\in \mathcal{P}_{[0,T+r]}}\|x_v-x_u\|^p\right)\\
            &= \sup_{\mathcal{P}_{[0,t_1]}}\left(\sum_{u,v\in \mathcal{P}_{[0,t_1]}}\|x_v-x_u\|^p\right)+\sup_{\mathcal{P}_{[t_1,t_1+r_1]}}\left(\sum_{u,v\in \mathcal{P}_{[t_1,t_1+r_1]}}\|x_v-x_u\|^p\right) + \dots \\
            &= \|\hat{x}\circ\tau\|^p_{p-\text{var};[0,T+r]} + \sum_{i=1}^\infty\|\hat{x}\|^p_{p-\text{var};S_i},
        \end{align*}
        where $S_i$ is the interval where $s\in[\tau(t_{i-}),\tau(t_i))$.

        It is now evident that the $p$-variation of the difference of the solutions to the canonical RDEs $y^1, y^2$ will be smaller than the difference of the solutions to the continuous RDEs $\tilde y^1, \tilde y^2$,  with equality holding when the paths along the $S_i$s are linear. More explicitly, we have
        \begin{equation*}
            \|y^1-y^2\|_{p-\text{var}}\leq\|\tilde{y}^1-\tilde{y}^2\|_{p-\text{var}} \lesssim |y^1_0-y_0^2| + \rho_{p-\text{var}}(\mathbf{x}^{\phi^1,\delta}_1,\mathbf{x}^{\phi^2,\delta}_2).
        \end{equation*}
        Taking the limit $\delta\rightarrow0$, we get our result. The fact that the flows are diffeomorphisms follows from \cite[Theorem 3.13]{chevyrev2019canonical}.
    \end{proof}
    \begin{proof}[Proof of Corollary \ref{thm:RSDE cadlag stability}]
        Consider the canonical rough flows
        \begin{align*}
            \Phi^1(t,x) &:= x + \int_0^tg(\Phi^1(s,x))\diamond d({\rough{\eta}}^1,\phi^1),\\
            \Phi^2(t,x) &:= x + \int_0^tg(\Phi^2(s,x))\diamond d({\rough{\eta}}^2,\phi^2).
        \end{align*}
        We know from Theorem \ref{thm:lip flow of canonical RDE} that the flows are Lipschitz, i.e.,
        \begin{equation*}
            \|\Phi^1(\cdot,x)-\Phi^2(\cdot,y)\|_{p-\text{var};[0,t]}\lesssim |x-y|+\beta_p({\rough{\eta}^1,{\rough{\eta}^2}}).
        \end{equation*}
        Since the flows are of continuous RDEs and are themselves diffeomorphisms, we have the inverse flows:
        \begin{equation*}
            ({\Phi^i})^{-1}(t,x) := \Psi^i(t,x) = x - \int_0^tg(x)\partial_x\Psi^i(s,x)\diamond d({\rough{\eta}^i},\phi^i).
        \end{equation*}
        Writing the equations \eqref{eq:canonical RSDE stability} as the composition of flows,
        \begin{equation*}
            X_t^i = \Phi^i(t,\Psi^i(t,X^i_t)),
        \end{equation*}
        and defining $\tilde{X}^i_t := \Psi^i(t,X_t^i)$, we have
        \begin{equation*}
            d\tilde{X}^i_t = \tilde{b}^i(\tilde{X}^i_t)dt+\tilde{\sigma}^i(\tilde{X}^i_t)\circ dB_t + \tilde{f}^i(\tilde{X}^i_t)\diamond d\xi_t.
        \end{equation*}
        Then, Theorem \ref{thm:lip flow of canonical RDE} implies
        \begin{equation*}
            \|X^1_\cdot-X^2_\cdot\|_{p-\text{var};[0,t]} =  \left\|\Phi^1(\cdot,\tilde{X}^1_\cdot)-\Phi^2(\cdot,\tilde{X}^2_\cdot)\right\|_{p-\text{var};[0,t]}\lesssim |\tilde{X}^1_t-\tilde{X}^2_t| + \beta_p(\rough{\eta}^1,\rough{\eta}^2).
        \end{equation*}
        We now have to bound the first term which can be done in the same way as in Theorem \ref{thm:rough-poisson-cont} since this is a special case.
    \end{proof}

\section{Other Desiderata}\label{app:desiderata}

\begin{theorem}[BDG Inequality]
    For any $1\leq p <\infty$ there exists positive constants $c_p,C_p$ such that for all local martingales $X_t$ and stopping times $\tau$, the following inequality holds
    \begin{equation}
        c_p\mathbb{E}[\langle X\rangle^{p/2}_\tau]\leq \mathbb{E}[\sup_{s\leq \tau}|X_s|]\leq C_p\mathbb{E}[\langle X\rangle^{p/2}_\tau],
    \end{equation}
    where $\langle X\rangle_t$ is the quadratic variation of the process.
\end{theorem}

\begin{theorem}[Gronwall]\label{thm:gronwall}
    Given the integral inequality
    \begin{equation}
        \phi(t) \leq K+L\int_0^t \phi(s)|dx_s|,
    \end{equation}
    where $x_t$ is a path with finite variation, then the following inequality holds
    \begin{equation}
        \phi(t) \leq K\exp(L|x|_{1-\text{var}:[0,t]}).
    \end{equation}
\end{theorem}

\begin{theorem}\label{thm:rde=sde}
    Let $M_t$ be a semi-martingale on the probability space $(\Omega,\mathcal{F},\mathbb{P})$ and $Y_t$ be a continuous semi-martingale on the probability space $(\bar{\Omega},\bar{\mathcal{F}},\bar{\mathbb{P}})$. Let $X_t$ be the solution to the doubly stochastic SDE
    \begin{equation}
        dX_t(\cdot,\bar{\omega}) = a(X_t(\cdot,\bar{\omega}))dt + b(X_t(\cdot,\bar{\omega}))\circ dM_t + c(X_t(\cdot,\bar{\omega}))\circ dY_t(\bar{\omega}), \hspace{25pt} X_0(\cdot,\bar{\omega}) = x(\cdot,\bar{\omega}),
    \end{equation}
    on the product space $(\hat{\Omega},\hat{\mathcal{F}},\hat{\mathbb{P}}) := (\Omega\times\bar{\Omega},\mathcal{F}\otimes\bar{\mathcal{{F}}},{\mathbb{P}}\otimes\bar{\mathbb{P}})$.
    Let $\rough{Y}(\bar{\omega}) := (Y(\bar{\omega}),\mathbb{Y}^{\text{strat}}(\bar{\omega}))$ be the Stratonovich rough lift of $Y(\bar{\omega})$. Let $S_t$ be the solution to the RSDE
    \begin{equation}
        dS_t(\cdot,\bar{\omega}) = a(S_t(\cdot,\bar{\omega}))dt + b(S_t(\cdot,\bar{\omega}))\circ dM_t + c(S_t(\cdot,\bar{\omega})) d\rough{Y}_t(\bar{\omega}), \hspace{25pt} S_0(\cdot,\bar{\omega}) = x(\cdot,\bar{\omega}).
    \end{equation}
    Then $\mathbb{P}$ a.s., $S_t = X_t$.
\end{theorem}

\begin{proof}
    Define the stochastic flow
    \begin{equation*}
        \Theta(t,x)(\bar{\omega}) := x + \int_0^tc(\Theta(s,x)(\bar{\omega}))\circ dY_s(\bar{\omega}).
    \end{equation*}
    and the rough flow
    \begin{equation*}
        \phi^{\rough{Y}}(t,x)(\bar{\omega}) = x + \int_0^t c(\phi^{\rough{Y}}(s,x)(\bar{\omega}))d{\rough{Y}}_s(\bar{\omega})
    \end{equation*}
    Since the stochastic flow is defined using the Stratanovich integral, the integral can be defined using smooth approximations of the process $Y_.$. The rough integral is defined using the Stratonovich lift of the path $Y_.(\bar{\omega})$, the rough path is a geometric rough path, which can also be constructed as the limit of smooth paths. \\
    Keeping this in mind, define the inverse stochastic and rough flows $\Theta^{-1}(t,x)$ and $(\phi^{\rough{Y}})^{-1}(t,x)$, and the processes
    \begin{align*}
        d\tilde{X}_t(\cdot,\bar{\omega}) &:= d\Theta^{-1}(t,X_t(\cdot,\bar{\omega})), \hspace{25pt} &X_0(\cdot,\bar{\omega}) = x(\cdot,\bar{\omega}),\\
        d\tilde{S}_t(\cdot,\bar{\omega}) &:= d(\phi^{\rough{Y}})^{-1}(t,S_t(\cdot,\bar{\omega})), &S_0(\cdot,\bar{\omega}) = x(\cdot,\bar{\omega}).
    \end{align*}
    Since geometric rough paths as well as Stratanovich integrals obey the same type of "chain rule", we have $\tilde{S}_.(\cdot,\bar{\omega}) = \tilde{X}_.(\cdot,\bar{\omega})$ a.s. if $\Theta(t,x)(\bar{\omega}) = \phi^{\rough{Y}}(t,x)(\bar{\omega})$. But we know that the flows are the same from the Universal Limit theorem of rough paths \cite[Theorem 9.2(b)]{friz2014course}.
\end{proof}

\begin{theorem}\label{thm:inv-flow-proof}
    Let $\eta_t$ be a continuous path on $\mathbb{R}$ and $c:\mathbb{R}\rightarrow\mathbb{R}$ be a Lipschitz function. Then the inverse of the ODE flow
    \begin{equation}
        \phi(t,x) := x + \int_0^tc(\phi(s,x))d\eta_s,
    \end{equation}
    is given by
    \begin{equation}
        \psi(t,x) = x - \int_0^tc(x)\partial_x\psi(s,x)d\eta_s.
    \end{equation}
\end{theorem}

\begin{proof}
    We show that $\psi(t,\phi(x,t)) = x$. Consider
    \begin{align}
        \frac{d}{dt}\psi(t,\phi(t,x)) &= \frac{\partial \psi}{\partial t}(t,\phi(t,x)) + \frac{\partial \psi}{\partial x}(t,\phi(t,x))\times\frac{d}{d t}\phi(t,x) \\
        &= -c(x)\partial_x\psi(t,\phi(t,x)) + \partial_x\psi(t,\phi(t,x))c(x)\\
        &= 0
    \end{align}
    But we also have $\psi(0,\phi(0,x)) = \psi(0,x) = x$. Thus, for all $t$,
        $\psi(t,\phi(t,x)) = x$.
\end{proof}

\end{appendix}

\end{document}